%% file: main.tex
\documentclass[%
 aip,
amsmath,amssymb,
reprint,%
onecolumn,%
secnumarabic
]{revtex4-1}

\usepackage{soul}

\usepackage{color}
\usepackage[margin=1.5in]{geometry}
\usepackage{graphicx}
\usepackage{dcolumn}
\usepackage{bm}
\usepackage{hyperref}
\usepackage{epstopdf}

\usepackage{amsthm}
\newtheorem{lem}{Lemma}
\newtheorem{thm}{Theorem}
\newtheorem{cor}{Corollary}
\newtheorem{rem}{Remark}

\begin{document}

\title[]{On Matching, and Even Rectifying, Dynamical Systems\\
through Koopman Operator Eigenfunctions.
}

\author{Erik M.~Bollt}
\email{bolltem@clarkson.edu}
\affiliation{Department of Mathematics, Department of Electrical and Computer Engineering, Department of Physics, Clarkson University, Potsdam, New York 13699, USA} 

\author{Qianxiao Li}
\email{liqix@ihpc.a-star.edu.sg}
\affiliation{Institute of High Performance Computing, Agency for Science, Technology and Research, Singapore 138632, Singapore} 

\author{Felix Dietrich}
\email{felix.dietrich@tum.de}
\affiliation{Department of Chemical and Biomolecular Engineering, Department of Applied Mathematics and Statistics, Johns Hopkins University and JHMI} 

\author{Ioannis
Kevrekidis}
\email{yannis@princeton.edu}
\affiliation{Department of Chemical and Biomolecular Engineering, Department of Applied Mathematics and Statistics, Johns Hopkins University and JHMI} 


\date{\today}

\begin{abstract}
Matching dynamical systems, through different forms of conjugacies and equivalences, has long been a fundamental concept, and
a powerful tool, in the study and classification of nonlinear dynamic behavior (e.g. through 
normal forms). 
In this paper we will
argue that the use of the Koopman operator and its spectrum is particularly well suited
for this endeavor, both in theory, but also especially in view of recent data-driven
algorithm developments.
We believe, and document through illustrative examples, that this can nontrivially extend the use and applicability of the
Koopman spectral theoretical and computational machinery beyond modeling and prediction,
towards what can be considered as a systematic discovery of  ``Cole-Hopf-type"
transformations for dynamics.

\end{abstract}

\maketitle

\smallskip
\noindent\textbf{Keywords: Koopman operator, rectification, conjugacy, flow box, DMD, EDMD, dynamical systems, data-driven algorithms} 


\section{Introduction}


A central concept in dynamical systems theory since the inception of the field (dating back to foundational work by Henri Poincar\'{e})  has been the concept of classifying systems ``up to" some notion of equivalence, \cite{mawhin1993centennial}, such as, for example, conjugacy of flows, that is, a homeomorphism between the state spaces such that the flows commute.

However, even if such a homeomorphism is known to exist, constructing it in closed form may be quite difficult.  Our own previous work~\cite{skufca2008concept} has focused on fixed point iteration methods to construct conjugacies when they exist or measure the defect from conjugacy when they do not exist, but the application was limited, and alternatively we tried symbolic dynamics methods~\cite{bollt2010comparing}.
In this paper, we will develop a methodology to tackle the problem of homeomorphism construction. 
Even though this will only be possible for us under significantly restrictive assumptions, we believe that the methodology
is nontrivial and informative, and we will find the homeomorphism explicitly in many examples.
We also introduce a computational method for the purely data-driven approximation of this mapping.
Our methodology is based on the fact that conjugacy---that is, existence of a \textit{spatial} isomorphism---between two dynamical systems directly implies the existence of a \textit{spectral} isomorphism between them. 
%
%
This spectral isomorphism relates the Koopman operators of the two systems~\cite{redei2012history}.

In the Koopman operator framework, the central objects of study are observables, which are functions of the state of the dynamical system. The action of the Koopman operator on these functions describes their temporal evolution, driven by the underlying dynamics. The operator acts linearly on the function space, a central property that makes it interesting for numerical approximation.
Typically, a certain number of eigenfunctions of the operator are obtained (in closed form or numerically), and then a linear combination of them is used to approximate a given observable. A particular observable of interest is the identity function on the state space, because the temporal evolution of this observable directly corresponds to the temporal evolution of different initial conditions.


The concept of ``matching" dynamical systems by the spectrum of their Koopman operators is not new.
In fact, it lies at the heart of ergodic theoretical developments, since the classic work of J. von Neumann and Paul Halmos, \cite{neumann1932operatorenmethode, halmos1942operator};  for a review see \cite{redei2012history}.
%
Yet the classical literature appears more concerned with proving the existence of the transformation, rather than with systematically obtaining it in closed or data-driven form. 
The notion of conjugacy for maps, and orbit equivalence for flows, is known to imply strong statements connecting the Koopman eigenfunctions of the conjugate systems, \cite{mezic-2017, lan2013linearization}; (Semi)-conjugacies (factors) also play an important role in Koopman spectral analysis {with the first such from data being} \cite{mezic2004comparison}.  In particular, many examples of nontrivial eigenfunctions can be, and have been, found given a conjugacy to a simpler system (linear) and a general nonlinear system when the homeomorphism connecting them may be known.  
{In particular, Lan and Mezic\cite{lan2013linearization} expanded the linearizing transformations of a nonlinear system to the full basin of attraction of an equilibrium. Their work extended the local homeomorphism beyond the usual neighborhood of the equilibrium, where it can be constructed using the famous Hartman-Grobman theorem \cite{perko2013differential}.
In \cite{mezic-2017} there are remarks relating this extension to the flow-box theorem. In \cite{mauroy-2013}, the authors discuss that Koopman spectral analysis relates isostables and isochrons to transformations to the classical action-angle decomposition of systems in the neighborhood of an attracting limit cycle.  Likewise, \cite{lan2013linearization} contains examples where the linearizing transformations are Koopman eigenfunctions with eigenvalue one. 
Recently, Mohr and Mezic\cite{mohr2016koopman} introduced a concept of principal eigenfunctions of a nonlinear system with an asymptotically stable hyperbolic fixed point. They developed a construction of a sequence of approximate conjugacies (in a finite-dimensional Banach space) which, in the spirit of the Stone-Weierstrass theorem, converges uniformly to a diffeomorphism transforming the nonlinear system (i.e., matching it) in a certain domain to a linear one.
Our work builds on these prior benchmarks. We expand on the concept that spectrally matching eigenfunctions may implicitly include information concerning the conjugacy connecting the state spaces. Given appropriately matched eigenfunctions (a concept that we define below), we might directly construct the conjugacy transformation between systems.}

In considering using the Koopman operator machinery towards such a construction, 
note that when the Koopman operator is restricted to an attractor (e.g. a quasi-periodic torus, or say, a Lorenz attractor), or is only considered in an ergodic setting, the restricted operator is unitary \cite{mezic-2005,budivsic2012applied}. 
%
Different types of ergodic systems, characterized by mixed, both point and continuous spectra of the operator  are also being studied\cite{das-2017}, in combination with numerical techniques to select smooth eigenfunctions\cite{giannakis-2015b}.
Yet for a quantitative analysis of more general dynamical system applications, it is
also important to be able to usefully deal with non-ergodic invariant sets.
Here, we focus on such sets and use them to construct 
transformations whose domain of validity will be defined below.



Instead of approximating a large number of eigenfunctions (necessary to improve {\em prediction accuracy} through their linear combinations), for \textit{matching purposes} we only make use of a specific subset of eigenfunctions for each system, in number equal to the dimension of the corresponding state spaces.
%
%
%
We use these two sets of eigenfunctions (assumed to be known {\em a priori}), to reconstruct the homeomorphism between the two systems.
Computational experiments along these lines have started to appear in the literature (involving the current authors); the first example arose in a 
data-fusion context\cite{williams-2015b}, and the second in a more general study of gauge invariance in data mining\cite{kemeth-2017}.


One requirement is that the eigenfunctions are selected in pairs, one per system, each pair associated to the same eigenvalue. 
However, it is not enough to require that the paired eigenfunctions are associated to the same eigenvalue---they also have to be related by the homeomorphism between the two flows. This would be trivial if a single eigenfunction were associated with each eigenvalue---but this is not the case, in general, as we will rationalize through the construction of the eigenfunctions through  a simple linear PDE which represents the infinitesmal operator.

Selection of the particular eigenfunctions constituting the two sets is, therefore, a difficult problem which we ``assume away''.
We do, however, point out a way to circumvent this problem by appropriately restricting the function space in which we assume they lie.
We discuss this possibility in detail, and present a framework of assumptions such that the selection of corresponding pairs of eigenfunctions can be based on matching eigenvalues alone.

%
%

As we will discuss, the computational implementation of our framework (which we will call ``matching EDMD", or EDMD-M)  is closely related to the EDMD method for which working algorithms are readily available\cite{williams-2015,williams-2015b,li-2017}.
We illustrate our EDMD-M framework in two ways: By finding, in closed form, a homeomorphism between a linear and a conjugate nonlinear, two-dimensional system; and numerically for the Van der Pol oscillator and a system conjugate to it. 


An obvious, yet exciting application of the matching concept is its relation to the classic topic of ``rectification", that is, 
matching to the constant vector field, $\dot{z}=(1,0,..,0)$, which is arguably simpler than a linear dynamical system.
In this sense, our approach realizes the well-known flow-box theorem, \cite{perko2013differential} 
also known as ``rectifying" the problem \cite{carmen2000ordinary, teschl2012ordinary}. 
We demonstrate this possibility by first (partially) matching a nonlinear system to a linear one, then (partially) matching the linear system to its rectification, and then composing the resulting homeomorphisms to rectify the nonlinear system.  

The remainder of the paper is organized as follows: 
Section~\ref{sec:background} briefly outlines the mathematical background of topological orbit equivalence, the Koopman operator, and its infinitesimal generator.
With these prerequisites, we define and prove the main contribution of the paper---an explicit construction (under, we repeat, specific assumptions) of the homeomorphism between topologically orbit equivalent systems in section~\ref{sec:matched keigs}.
We discuss the nontrivial geometric multiplicity of eigenvalues of the Koopman operator in section~\ref{defect}, and then show how the  function space of observables can be restricted appropriately so as to circumvent this multiplicity of the spectrum.
The use of the ideas in the paper is demonstrated throughout via analytical as well as numerical examples.


\section{Background}\label{sec:background}

In this section, we review some standard theory of dynamical systems, leading to the Koopman operator and its infinitesimal generator as a PDE, the proof of which is in Appendix \ref{pdeforKoopmaneig}, with examples in Appendix \ref{solns}.  This section serves not only to recall important, albeit standard material,  but also it allows us to cast this standard language for our specific interest of matching systems through spectral concepts.  A thoroughly familiar reader may wish to skip forward to Sec.~\ref{sec:matched keigs}.

\subsection{Review of Matching Systems, Conjugacy, and Orbit Equivalence}\label{sec:review matching}

A central problem in dynamical systems is the notion of equivalence, standard definitions of which are topological in nature. Two discrete time maps, 
$F^{(1)}:{\mathbb X} \rightarrow {\mathbb X}$, and $F^{(2)}:{\mathbb Y}\rightarrow {\mathbb Y}$ are called conjugate if there is a homeomorphism (a regular change of coordinates) between ${\mathbb X}$, and ${\mathbb Y}$, that is $y=h(x):{\mathbb X}\rightarrow {\mathbb Y}$, such that I: $h$ is a homeomorphism between topological spaces ${\mathbb X}$, and ${\mathbb Y}$,  (that means $h$ is (1) one-to-one, (2) onto, (3) continuous, (4) the inverse is continuous, and together this defines the spaces as topologically equivalent), and II: the maps commute with respect to the change of variables; formally stated, $h\circ F^{(1)}(x)=F^{(2)}\circ h(x)$ for all $x\in {\mathbb X}$.  If furthermore, $h$ and $h^{-1}$ are continuously differentiable, then the maps are called diffeomorphic, which is a stronger form of equivalence.
When stating $F^{(i)}$, the superscript refers to the $i$-th system. Subscript $x_i$ denotes the $i$-th coordinate.

When discussing continuous time systems, the notion of topological orbit equivalence is relevant.
 If two (semi)flows $S_t^{(1)}:{\mathbb X} \rightarrow {\mathbb X}$ and $S_t^{(2)}:{\mathbb Y}\rightarrow {\mathbb Y}$ are topologically orbit equivalent, there exists a homeomorphism, $h:{\mathbb X}\rightarrow {\mathbb Y} $ such that $S_t^{(2)}(h(x))=h(S_t^{(1)}(x))$ for every $t\geq 0$ (or for every $t\in {\mathbb R}$ if they are each flows).

Conjugacy between dynamical systems defines an equivalence relation between vector fields on a manifold: two vector fields $F^{(1)},F^{(2)}$ on spaces (manifolds) ${\mathbb X}$ and ${\mathbb Y}$ respectively are equivalent if and only if their induced dynamical systems $\dot{x}=F^{(1)}(x)$ and $\dot{y}=F^{(2)}(y)$ are topologically orbit equivalent.  When such a topological orbit equivalence $h:{\mathbb X}\rightarrow {\mathbb Y} $ exists between the induced flows, we call the two systems ``matched''. 

\subsection{Review of the Koopman operator, and its Eigenfunctions}\label{koopstruct}

Here we recall the definition for the {\it Koopman operator structure} of a flow.
Consider, for example, a differential equation in $\mathbb{R}^d$, 
\begin{equation}\label{ode}
\dot{x}=F(x),
\end{equation}
 associated with a (here, autonomous) vector field,
\begin{equation}\label{ode2}
F:{\mathbb R}^d\rightarrow {\mathbb R}^d.
\end{equation}
 Note that a nonautonomous problem, $f(x,t):{\mathbb R}^d \times {\mathbb R} \rightarrow {\mathbb R}^d$ can be written in $d+1$ dimensions as an autonomous problem by augmenting with a time variable $\tau=t$, appending $\dot{\tau}=1$ to $f$ so that $F=[f,1]^T$. 
  Depending on $F$,  we assume an invariant flow\cite{perko2013differential} 
$S_t:M\to M$ for each $t\in\mathbb{R}$ (or semi-flow for $t\geq 0$), where $M\subseteq {\mathbb R}^d$ denotes an invariant set. We write $x(t):=S_t(x_0)$ for a trajectory starting at $x(0)=x_0\in M$.

The associated Koopman operator (also called composition operator) describes the evolution of ``observables", or ``measurements" along the flow, \cite{budivsic2012applied, mezic2013analysis}. These ``observation functions" $g:M\rightarrow {\mathbb C}$ are elements of a space of observation functions ${\cal F}$, for example ${\cal F}=L^2(M)=\{g:\int_M |g(s)|^2 ds<\infty \}$, which is commonly useful in numerical applications that utilize the inner product associated with the Hilbert space structure \cite{mezic-2005,budivsic2012applied,giannakis-2015b,kutz2016dynamic,mezic-2017}.
The Koopman operator ${\cal K}_{S_t}$ is associated with the flow $S_t$, such that
\begin{equation}\label{eq:koopman operator}
{\cal K}_{S_t}[g](x)=g\circ S_t(x).
\end{equation}
That is, for each $x$, we observe the value of an observable $g$ not at $x$, but ``downstream" by time $t$, at $S_t(x)$.  This defines the operator
as a flow ${\cal K}_{S_t}=:{\cal K}_t:{\cal F}\to {\cal F}$ on the function space ${\cal F}$ for each $t\in\mathbb{R}$ (or as a semi-flow if the relation only holds for $t\geq 0$).
An  interesting feature of the Koopman operator is that it is linear on $\mathcal{F}$, but at the cost of being infinite dimensional, even though the flow $S_t$ may be associated to a finite dimensional and nonlinear vector field (Eqs.~(\ref{ode})-(\ref{ode2})).
The spectral theory of Koopman operators \cite{gaspard1995spectral, budivsic2012applied, mezic2013analysis} concerns eigenfunctions and eigenvalues of the operator ${\cal K}_t$, which may be stated  in terms of the equation
\begin{equation}\label{eq:koopman definition}
{\cal K}_t[g](x)=b^t g(x)=e^{\lambda t} g(x).
\end{equation}
See Fig.~\ref{fig1} for a visualization of the action described in Eq.~\ref{eq:koopman definition}, and the discussion in the Appendix~\ref{pdeforKoopmaneig}.
\begin{figure}[hpt]     
	\centering   
			\def\svgwidth{.45\textwidth}
        	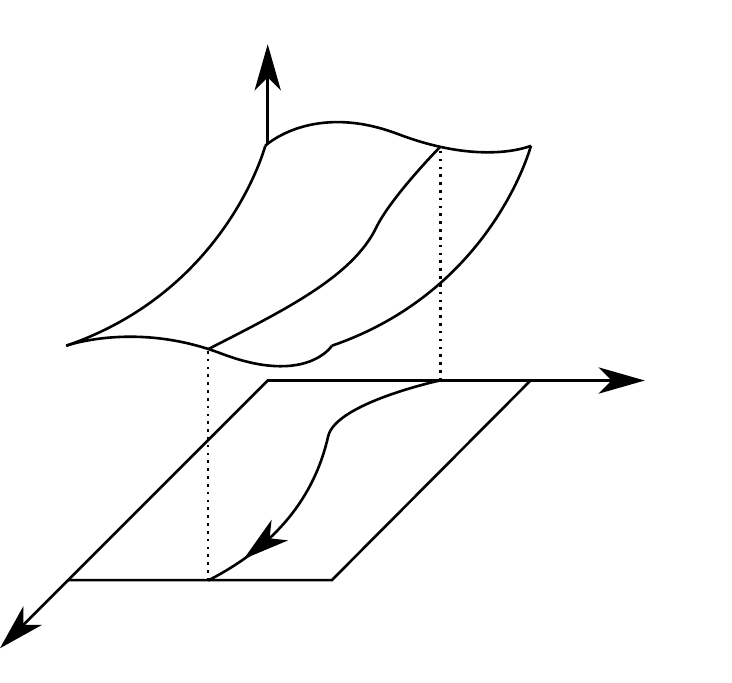
        	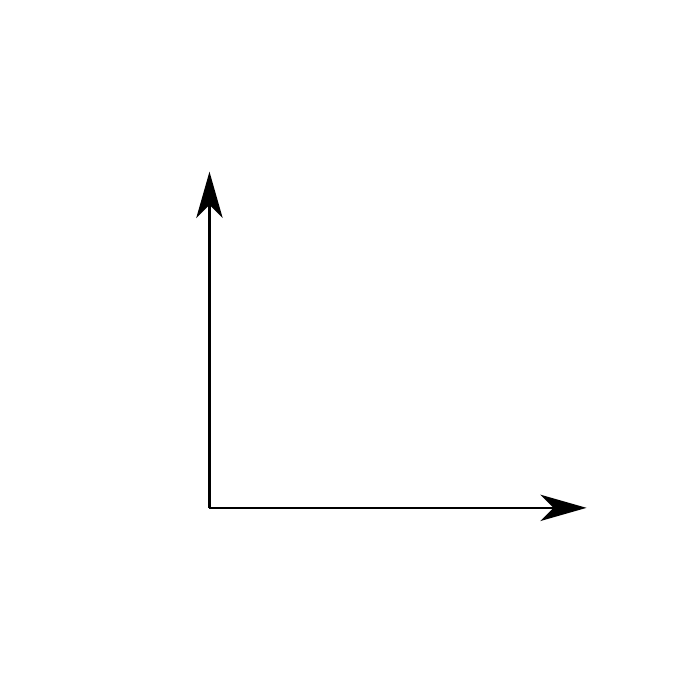
       	\caption{ Illustration of the action of a Koopman operator ${\cal K}_t$ corresponding to a flow $z(t)=S_t(z_0)$ as the observation of a function $g(z)$ along an orbit.  (Left) The notation ${\cal K}_t[g](z)=g\circ \phi_t(z)$ denotes that we measure $g$ at $\phi_t(z)$, $t$ time units downstream on the trajectory initialized at $z$. An eigenfunction $f$ satisfies the equation  ${\cal K}_t[f](z)=e^{\lambda t} f(z)$, implying that $f$ also satisfies the linear equation (Eq.~(\ref{n1})).   }
	\label{fig1}
\end{figure}
We will write an eigenvalue---eigenfunction pair of the Koopman operator as $(\lambda,g_\lambda(x))$, and call the pair a ``KEIG". For convenience, we will say ``KEIGs" either when referring to Koopman eigenfunctions or to the eigenvalue and eigenfunction pairs. We shall assume that $g_\lambda\in{\cal F}$.  The multiplying factor may also be written in terms of the eigenvalue, $b=e^\lambda$.


Despite the importance and popularity of the Koopman operator in an ever broadening theoretical, computational and applied literature in the field of dynamical systems, \cite{mezic2013analysis, budivsic2012applied, kutz2016dynamic}, there are surprisingly few explicit example dynamical systems, meaning explicit vector fields, for which KEIG pairs $(\lambda,g(z))$ can be explicitly written, analytically, in closed form.  There is the linear case, that we now state, but beyond that, we are not aware of many other examples, other than those produced through conjugacy to the linear system.  Instead, the literature emphasizes descriptive statements about the role and importance of the KEIGs, and then numerical methods, notably either  the matrix methods called DMD, \cite{schmid-2010,kutz2016dynamic} (and variants such as EDMD, \cite{williams-2015,williams-2015b,williams-2014}, including EDMD-DL, \cite{li-2017}) are used.  Our goal in this section is to note how the infinitesimal generator of the Koopman operator defines a first order PDE whose solutions are eigenfunctions.  Note that this fact has also been recently utilized in \cite{kaiser-2017} as a central element in developing control laws.

The adjoint of the Koopman operator is the Perron-Frobenius operator\cite{mezic-2005} (or transfer operator).
Several methods for the approximation of the transfer operator through its eigenfunctions have been developed, including the use of its infinitesimal generator\cite{klus-2017}. The same holds for the Koopman operator\cite{mezic2013analysis,budivsic2012applied,williams-2015,williams-2014,williams-2015b} and its generator\cite{kaiser-2017}.
Transfer and Koopman operators are often studied in the ergodic or stochastic setting.
Given an Ito diffusion process with a smooth vector field as the drift term, the infinitesimal generator of the transfer operator defines the \textit{Kolmogorov forward equation} (or \textit{Fokker-Planck equation}), whereas the \textit{Kolmogorov backward equation} is related to the infinitesimal generator of the Koopman operator \cite{klus-2017}. 

The linear problem is the example with an explicit KEIGs solution that is often discussed, where
\begin{equation}\label{linear1}
\dot{z}=F(z)=a z,
\end{equation}
 as a special case of Eq.~(\ref{ode}) where $a, z(t)\in {\mathbb R}$.  The KEIG pairs of this system may be described as the state observer, 
 \begin{equation}\label{KEIGSlin1}
 (\lambda,g(z))=(a,z),
 \end{equation}
  meaning the identity function 
  and the growth factor serve as a KEIG.
  This is easy to directly confirm, using the solution of the flow,
$z(t)=S_t(z_0)=e^{a t}z_0$.  It follows that,
\begin{equation}\label{check1}
{\cal K}_t[g](z)=g \circ S_t(z)=e^{a t}z=b^t g(z).
\end{equation}
The eigenvalue $\lambda=a$ also provides the multiplying factor, $b=e^\lambda$.
Likewise in a multivariate linear case, if $z(t)\in {\mathbb R}^d$ then we use the notation, $[\cdot]_i:{\mathbb R}^d\rightarrow {\mathbb R}$ to denote the projection function that selects the $i${-th} component of $z$.  Then if $\dot{z}=F(z)=A z,$ and $A=\Sigma=diag({a_i})$ is a diagonalized matrix, then $(a_i,[z]_i)$ are KEIGs.  The general linear problem is handled accordingly by standard linear theory, by considering similarity transformations to Jordan block form \cite{golub2012matrix}; {see \cite{mezic-2017} for details and definition of generalized eigenfunctions.}



We now describe the simple linear PDE, whose solutions are eigenfunctions of the Koopman operator (for a use case of the PDE in control of dynamical systems, see \cite{kaiser-2017}). This will help us to proceed toward the goal of this paper, which is to use KEIGs to develop the conjugacy between related systems.

\begin{thm}\label{thm2}
 Given a domain ${\mathbb X}\subseteq M\subseteq{\mathbb R^d}$, $z\in{\mathbb X}$, and $\dot{z}=F(z)$ with $F:{\mathbb X}\rightarrow {\mathbb R}^d$, then the corresponding Koopman operator has eigenfunctions $g(z)$ that are solutions of the linear PDE,
\begin{equation}\label{qpde}
\nabla g \cdot F(z)=\lambda g(z),
\end{equation}
if $\mathbb{X}$ is compact and $g(z):{\mathbb X}\rightarrow {\mathbb C}$ is in $C^1({\mathbb X})$, or alternatively, if $g(z)$ is $C^2({\mathbb X})$. 
\end{thm} 
The proof of this Koopman PDE theorem by discussion of infinitesimal generators is in Appendix \ref{pdeforKoopmaneig}, along with a statement concerning the formulation for weak solutions.  Explicit examples of Koopman eigenfunctions as solutions of this PDE are given in Appendix \ref{solns}.

\section{Constructing Orbit Equivalence from Eigenfunctions}\label{sec:matched keigs}

Having established the mathematical background for Koopman operators and topological orbit equivalence in the previous section, we will now describe the main contribution of the paper: an approach to the construction of the transformation function between two (already assumed) conjugate systems. We proceed by stating the commonly used relationship between eigenfunctions of conjugate systems in Theorem~\ref{thm1}. Then, we argue that the transformation can be recovered if ``matching'' pairs of eigenfunctions are given. We formulate this in two corollaries, Cor.~(\ref{cor1}) for the one-dimensional case and Cor.~(\ref{cor2}) for higher dimensions.

The reason for the distinction between the two corollaries is that in higher dimensions, the subspace spanned by eigenfunctions associated to a given eigenvalue can be larger than one-dimensional.
%
We refer to this property of the Koopman operator as ``nontrivial geometric multiplicity", in analogy to linear algebra of matrices. The multiplicity results in a significant complication to our goal of matching systems, because it is no longer clear which eigenfunction in one system ``matches'' a particular eigenfunction of the other system---equality of the associated eigenvalues is not enough.

In all of the examples and statements below, we focus on systems with nonzero, real eigenvalues. In case a statement is valid for a larger set of eigenvalues, for example, the whole complex plane, we will state it explicitly.
The following theorem\cite{budivsic2012applied} describes the relationship between eigenfunctions of topologically equivalent flows.  The related theorem for topologically conjugate maps is found for example in, \cite{mezic2013analysis, budivsic2012applied}.  {A stronger statement of the following, including treatment of generalized eigenfunctions is found as Proposition 3.1 in \cite{mezic-2017}.}

\begin{thm}\label{thm1}
If two (semi)flows $S_t^{(1)}:{\mathbb X} \rightarrow {\mathbb X}$ and $S_t^{(2)}:{\mathbb Y}\rightarrow {\mathbb Y}$ are topologically orbit equivalent by $h:{\mathbb X}\rightarrow {\mathbb Y}$,  
and  each of these two flows has a Koopman operator structure as reviewed in Sec.~\ref{koopstruct}, and if $g_{\lambda}^{(2)}(x)$ is a Koopman eigenfunction of ${\cal K}_t^{(2)}$ associated with $S_t^{(2)}$ and with an eigenvalue $\lambda$, then,
\begin{equation}\label{relate}
g_{\lambda}^{(1)}(x)=g_{\lambda}^{(2)}\circ h(x),
\end{equation} is an eigenfunction associated with the Koopman operator ${\cal K}_t^{(1)}$ of $S_t^{(1)}$, and is associated with the same eigenvalue $\lambda$. 
\end{thm}

A proof of this is found in \cite{budivsic2012applied}, which we restate here in our notation, 
\begin{eqnarray}
e^{\lambda t}g_{\lambda}^{(2)}\circ h (x)&=& \nonumber \\
&=&e^{\lambda t} g_{\lambda}^{(2)}(y)
={\cal K}_t^{(2)}[g_{\lambda}^{(2)}](y)
=g_{\lambda}^{(2)}\circ S_t^{(2)} (y)
=g_{\lambda}^{(2)}\circ S_t^{(2)} \circ h(x)
=g_{\lambda}^{(2)}\circ h \circ S_t^{(1)} (x)= \nonumber \\
&=&{\cal K}_t^{(1)}[g_{\lambda}^{(2)}\circ h](x).
\end{eqnarray}
\medskip

Thm.~\ref{thm1} is stated as if the homeomorphism (the function $h$), is known.  
Additionally, it is assumed that \textit{some} KEIG $(\lambda,g_{\lambda}^{(1)}(x))$ for the ${\mathbb X}$-system is also known.  Then, the matching KEIG $(\lambda,g_{\lambda}^{(2)}(y))$ for the ${\mathbb Y}$-system is given by solving Eq.~(\ref{relate}) for $g_{\lambda}^{(2)}$.

In the following two sections (\ref{sec:1d cor1}) and (\ref{defectsection}), we flip what is usually considered known and unknown in Thm.~\ref{thm1}.
If we know the matching eigenfunctions $g^{(1)}$ and $g^{(2)}$ (or two matching \textit{sets} of eigenfunctions in higher dimensions, see Sec.~\ref{defectsection}), we can recover the transformation $h$. 

\subsection{Constructing Orbit Equivalence in the Real Line}\label{sec:1d cor1}
In this section, we assume knowledge of two matching eigenfunctions, one for each system, associated with the same eigenvalue, and by Thm.~\ref{thm1} with some unknown $h$. This information can be used to explicitly construct the change of variables $h$ between the two systems. We formulate this construction in the following corollary as a local statement at a point $x\in\mathbb{X}$, and then discuss the domain of validity around that point.

\begin{cor}\label{cor1} Consider two systems with states in invariant subsets ${\mathbb X},{\mathbb Y}\subseteq {\mathbb R}$, and with continuous time flows  $S_t^{(1)}$ and $S_t^{(2)}$ that are topologically orbit equivalent through an (unknown) homeomorphism $h$. Assume one Koopman eigenfunction $g^{(i)}_\lambda$, $i=1,2$, is given for each system, associated to identical, nonzero eigenvalues, $\lambda^{(1)}=\lambda^{(2)}$. Also assume that the eigenfunctions are related at a given point $x\in\mathbb{X}$ by $h$ such that $g_{\lambda}^{(1)}(x)=(g_{\lambda}^{(2)}\circ h)(x)=g_{\lambda}^{(2)}(y)|_{y=h(x)}$.
If $g_{\lambda}^{(2)}$ has an inverse, the transformation between the spaces may be written,
\begin{equation}\label{matcheq}
y=h(x)=g_{\lambda}^{(2),-1}\circ g_{\lambda}^{(1)}(x).
\end{equation}
Alternatively, by the implicit function theorem, if the Jacobian (in one dimension: the derivative) $J=\frac{\partial}{\partial x} g^{(2)}$ is invertible at the point $y=h(x)$, there exists an open set $A\subset{\mathbb X}$ with $x\in A$ such that the transformation $y=h(x)$ exists for all $x\in A$, and thus Eq.~\ref{matcheq} holds on $A$.
This transformation will have as much regularity at each $x \in {\mathbb X}$ as $g_{\lambda}^{(1)}$, and $g_{\lambda}^{(2)}$.
\end{cor}

Note that in particular, to use $h$ as we intend, i.e. to match \textit{differential equations}, a local \textit{diffeomorphism} is required.
Cor.~\ref{cor1} of Thm.~\ref{thm1} applies only locally around a given point $x$. We can extend the domain of the constructed function $h$ by the following arguments:
Denote the domain and range of $g^{(1)}$ \begin{equation}
D_1=D(g^{(1)}_\lambda,{\mathbb X})=\{x:x\in {\mathbb X}, g^{(1)}_\lambda:X\rightarrow {\mathbb C} \mbox{ exist} \}
\end{equation} 
\begin{equation}\label{range1}
R_1=R(g^{(1)}_\lambda,D_1,{\mathbb C})=\{z|z=g^{(1)}_\lambda(x), x\in D_1, z\in {\mathbb C}\},
\end{equation}
and likewise $D_2\subset {\mathbb Y}\subset {\mathbb R}$ and $R_2$ for $g^{(2)}_\lambda$.  
Also relevant is the preimage of $R_2\cap {\mathbb R}$ with respect to $g^{(1)}_\lambda$, 
\begin{equation}
g^{(1),-1}_\lambda
(R_2\cap {\mathbb R})=\{x|z=g^{(1)}_\lambda(x), z\in(R_2\cap {\mathbb R}), x\in {\mathbb X}\}.
\end{equation}
Then with this notation, the domain of $h$ in ${\mathbb X}$ for Eq.~(\ref{matcheq}) may be written  as 
\begin{equation}\label{intersects}
D(h,{\mathbb X})=D(g_{\lambda}^{(2),-1}\circ g_{\lambda}^{(1)},{\mathbb X})=g^{(1),-1}_\lambda
(R_2\cap {\mathbb R}).
\end{equation}
Note that we may restrict ${\mathbb X}$ to the subset $D(h,{\mathbb X})$, so that the change of variables $h$ will be defined on that subset of the domain (whenever it is not empty). 
This domain may not be invariant under the flow. In this case, it is possible to discuss eigenfunctions in terms of the corresponding subflows and subdomains\cite{mezic-2017}. 

\subsubsection{Example 1: Linearizing a Nonlinear Problem in 1D}

As an illustrative problem, we attempt to linearize (that is, to match to a linear system),
\begin{equation}\label{quad}
\dot{x}=F^{(1)}(x)=x^2, \quad x(0)=x_0.
\end{equation}
A function $h$ transforming system (\ref{quad}) to a linear system is only valid for $x\in\mathbb{R}/\{0\}$---possibly better thought of as the two domains $\mathbb{R}^+$ and $\mathbb{R}^-$. These two domains will be derived below, with a reference to Eq.~(\ref{intersects}), as soon as the eigenfunctions are available.
Equation~(\ref{quad}) illustrates the classic scenario of blow-up in finite time, since the solution flow is
\begin{equation}
x(t)=S_t^{(1)}(x_0)=\frac{1}{\frac{1}{x_0}-t}.
\end{equation}
The blow-up time is
\begin{equation}
t^*=\frac{1}{x_0},
\end{equation}
and this is called a movable singularity, since it depends on the initial condition.  The differential equation whose solutions are Koopman eigenfunctions $g_\lambda$ (from  Eq.~(\ref{thm1b})) is
\begin{equation}
\nabla g_\lambda \cdot F^{(1)}(x)=\frac{dg_\lambda}{dx} x^2=\lambda g_\lambda(x).
\end{equation}
By integrating factors, we find a general solution,
\begin{equation}\label{quadkeig}
g_\lambda(x)=c e^{-\frac{\lambda}{x}}
\end{equation}
for a constant $c=g_0e^{\frac{\lambda}{x_0}}$ that we can neglect since it appears on both sides of the equation.  Note that a related quadratic problem and its eigenfunctions were  described by Mezic\cite{mezic-2017}, where the derivation was through transformations found from center-manifold theory as a classic example in Kelley, \cite{kelley1967stable}.

We see that the $g_\lambda$ are indeed KEIGs by definition (\ref{eq:koopman definition}):
\begin{equation}
{\cal K}_t[g_\lambda](x)={\cal K}_t[e^{-\frac{\lambda}{x}}]=e^{-\lambda (S_t(x))^{-1}}=e^{-\lambda (\frac{x}{1-x t})^{-1}}=e^{\lambda (t-\frac{1}{x})}=e^{\lambda t}e^{-\frac{\lambda}{x}}=e^{\lambda t}g_\lambda(x),
\end{equation}
for any $\lambda\in\mathbb{C}$.
In Fig.~(\ref{fig2})(Right) we show the behavior of this KEIG function, for $\lambda=1$.
\begin{figure}[hpt]     
	\centering   
        	\includegraphics[width=.4\textwidth]{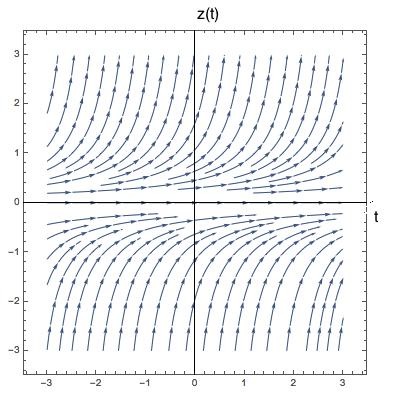}
        	\includegraphics[width=.5\textwidth]{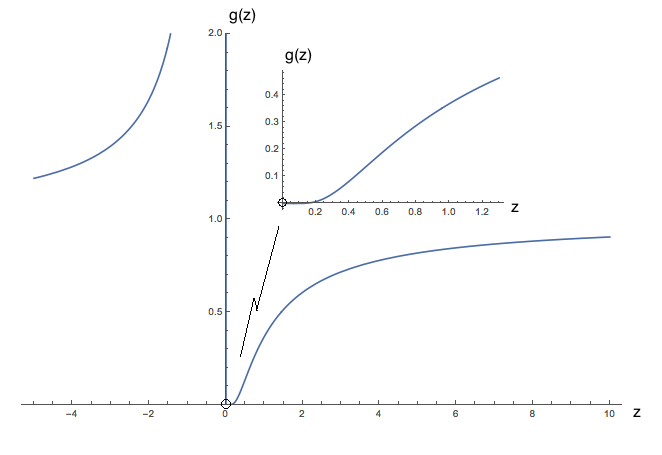}
       	\caption{ 
        (Left) Typical trajectories of the ODE Eq.~(\ref{quad}) where $F^{(1)}(z)=z^2$ leads to blow-up in finite time at $t^*=\frac{1}{z_0}$. (Right) The eigenfunctions $g_\lambda(z)=e^{-\frac{\lambda}{z}}$, one in each half-line, chosen with $\lambda=1$ from Eq.~(\ref{quadkeig}) are plotted (also blown up in the inset near $z=0$). 
The domain is best described as the two domains ${\mathbb R}^-$ and ${\mathbb R}^+$ with corresponding ranges $g(z)\in (1,\infty)$ and $(0,1)$. The initial condition selects the relevant domain.
}
	\label{fig2}
\end{figure}
It turns out for our purpose of matching systems, the choice of $\lambda$ is not important here, as long as it is nonzero and has the same value for both systems.
Nevertheless, the choice of $\lambda$ influences whether the eigenfunction is complex- or real-valued, and also imposes restrictions on the domain of the associated eigenfunction. We choose $\lambda$ to equal the parameter $a$ of the linear system we match to, which is real and positive, and here w.l.o.g. is set to 1. With this choice, the domain of definition of the transformation will be $\mathbb{R}/\{0\}$ (the union of $\mathbb{R}^+$ and $\mathbb{R}^-$), as shown below.



%
%

The goal of the example is to find a transformation $h$ from states of the nonlinear problem to states of the linear problem,
\begin{equation}
\dot{y}=F^{(2)}(y)=y,\  y\in\mathbb{R}
\end{equation}
which has a solution as the flow, 
\begin{equation}
y(t)=S_t^{(2)}(y_0)=e^{t}y_0.
\end{equation}
In this case, the KEIG differential equation specializes to Eq.~(\ref{qpde}),
\begin{equation}
\nabla g \cdot F^{(2)}(y)=\frac{dg}{dy} y =\lambda^{(2)} g(y).
\end{equation}
There is a general solution that can be found by integrating factors, $g(y)=k y^{\lambda^{(2)}}$, with $k=g_0x_0^{-\lambda^{(2)}}$,
 but by similar arguments we may choose,
\begin{equation}
g^{(2)}(y)=y^{\lambda^{(2)}},
\end{equation}
and note that this corresponds to what is called the ``state observer"  if $\lambda^{(1)}=\lambda^{(2)}=1$ for the sake of simplicity in defining the inverse of the resulting state observer: if $z=g^{(2)}(y)=y$, then the inverse is $y=g^{(2),-1}(z)=z$, the identity function. 
Considering Corollary~\ref{cor1}, the range $R_2$ of $g^{(2)}$ is $R_2=\mathbb{R}$, the domain of $g^{(1)}$ is $(\mathbb{R}^+\cup \mathbb{R}^-)$, and hence, following the definition of the domain for the transformation $h$ in Eq.~\ref{intersects}, we have
\begin{equation}
D(h,{\mathbb X})=g^{(1),-1}_\lambda(R_2\cap {\mathbb R})=g^{(1),-1}_\lambda( {\mathbb R})=(\mathbb{R}^+\cup \mathbb{R}^-).
\end{equation}
Therefore, the linearization of system~(\ref{quad}) is valid for all positive values of $x(t)$ (with time derivative defined through Eq.~\ref{quad}), and also for all negative values of $x(t)$. The value $0$ where the mapping is undefined, is excluded. 
By Theorem \ref{thm1} and Corollary \ref{cor1}, if there is a homeomorphism $y=h(x)$, then in each of the two domains,
\begin{equation}\label{lex}
y=h(x)=g^{(2),-1}\circ g^{(1)}_\lambda(x)=g^{(1)}_\lambda(x)=e^{-\frac{\lambda}{x}}.
\end{equation}
The KEIG $g^{(1)}$ from the nonlinear problem is, therefore, the transformation function $h$ that linearizes the system in each domain;
 by changing the variables,  the nonlinear system (with variable $x$) becomes the linear system (with variable $y$),
\begin{eqnarray}
\dot{y}&=&Dh(x)\dot{x}=Dh( h^{-1}(y))F^{(1)}(h^{-1}(y)) =\nonumber \\
&=&\left(\frac{e^{-\frac{\lambda}{x}}}{x^2}\right)\biggr\rvert_{x=h^{-1}(y)}(x^2)|_{x=h^{-1}(y)}
=\left(\frac{e^{-\frac{\lambda}{x}}}{x^2}\right)\biggr\rvert_{x=\frac{-\lambda}{\log y} }(x^2)\rvert_{x=\frac{-\lambda}{\log y}  }
=e^{-\frac{\lambda}{x}}\rvert_{x=\frac{-\lambda}{\log y} }=  \nonumber \\
&=&y.
\end{eqnarray}
While this is an easy transformation to obtain (or maybe even to guess!) in each of the two half-lines , in its way it {\em is} a "Cole-Hopf-type" transformation, even if only
for a scalar ODE:  it takes a nonlinear problem to a linear one away from singularities.

\subsubsection{Example 2: Rectifying a Nonlinear Problem in 1D}
We again use the quadratic ODE from Eq.~(\ref{quad}) as our ${\mathbb X}$-system.
To rectify this problem, we need to match to the ${\mathbb Y}$-system,
\begin{equation}
\dot{y}=F^{(2)}(y)=1.
\end{equation} 
After the eigenfunctions are available, we will find below that $D=\mathbb{R}^+$ through Eq.~(\ref{intersects}). 
For the $\mathbb{Y}$-system, the flow is
\begin{equation}
y(t)=S_t^{(2)}(y_0)=t+y_0.
\end{equation}
The KEIG differential equation in this case becomes
\begin{equation}
\nabla g \cdot F^{(2)}(y)=\frac{dg}{dy} 1 =\lambda g(y),
\end{equation}
and there is a solution, $g_{\lambda^{(2)} }^{(2)} =c e^{\lambda^{(2)} x}$, $c=  e^{-\lambda^{(2)} x_0}{g_0}$. Again, we can neglect the constant $c$, hence  
\begin{equation}
g^{(2)}_{\lambda^{(2)} }(y)=e^{\lambda^{(2)} y}.
\end{equation}
For real valued $\lambda^{(2)}$ and $y$, the function $g^{(2)}$ has range $R_2=\mathbb{R}^+$.
The inverse of $z=g^{(2)}_{\lambda^{(2)} }(y)$ is $y=g^{(2),-1}_{\lambda^{(2)} }(z)=\frac{\ln z}{\lambda^{(2)} }$, which is only defined on the domain $D_1=\mathbb{R}^+$. This domain of the inverse of $g^{(2)}_\lambda$, together with the range $R_2$ yields the domain $D=\mathbb{R}^+$ of $h$ (see Eq.~\ref{intersects}). 
Then, 
\begin{equation}\label{eq: 33}
y=h(x)=g^{(2),-1}_{\lambda_2}\circ g^{(1)}_{\lambda^{(1)} }(x)=\frac{\ln\left(e^{\frac{-\lambda^{(1)} }{x}}\right)}{\lambda^{(1)} }=\frac{-\lambda^{(1)} }{\lambda^{(2)}  x },
\end{equation}
and,
\begin{eqnarray}
\dot{y}&=&Dh(x)\dot{x}=Dh( h^{-1}(y))F^{(1)}(h^{-1}(y)) \nonumber \\
&=&\frac{\lambda^{(1)} }{\lambda^{(2)} }\frac{1}{x^2}|_{x=h^{-1}(y)} (x^2)|_{x=h^{-1}(y)} \nonumber \\
&=&\frac{\lambda^{(1)} }{\lambda^{(2)} }=1,
\end{eqnarray}
because we assumed that the eigenvalues are equal (Cor.~\ref{cor1}).
In summary, the system can be rectified for $x\in D=\mathbb{R}^+$, so that $\dot{y}=1.$
It is easy to see that if we defined the inverse of $g^{(2)}$ in Eq.~(\ref{eq: 33}) for negative $z$ instead, the system can be rectified on $\mathbb{R}^-$. 
This result is in agreement with the flow-box theorem, \cite{perko2013differential, teschl2012ordinary, carmen2000ordinary, marsden2006texts}. Here, we have found the transformation explicitly through matching KEIGs.

\subsection{Constructing Orbit Equivalence in More Than One Dimension: Problems with Nontrivial Geometric Multiplicity}\label{defectsection}

A central requirement for matching systems through Cor.~(\ref{cor1}) 
is a systematic way to select the appropriate eigenfunctions, one pair for each eigenvalue. In dimensions greater than one, 
 this selection can be difficult because of the nontrivial geometric multiplicity. In general, eigenfunctions are unique up to an equivalence class of scalar, complex multiples. In the univariate setting, every eigenvalue is associated to a single equivalence class of eigenfunctions. However, in the multivariate setting, many such classes of eigenfunctions may be associated to a single eigenvalue, so that the eigenspace spanned by the eigenfunctions has a dimension larger than one.

This can be rationalized by considering the construction of the eigenfunctions through the KEIGs PDE.
In the univariate case, prescribing a multiple of the initial condition at a point will give rise to another eigenfunction within the trivial (scalar multiple) equivalence class, due to linearity of the PDE.
 In the multivariate setting, 
the KEIGs PDE depends on initial data $g_0(z)$ defined on a co-dimension-one set $\Sigma$, which has to be chosen so that the flow $S_t$ of the ODE is transverse to it. (Reminder: we reiterate here that a method to construct the set numerically is described in Appendix~\ref{sec:construction by numcont};
And it is standard, as reviewed in Appendix~\ref{solns}, that the solution curves $z(t)=S_t(z_0)$ of $\dot{z}=F(z), z_0\in \Sigma$, serve as characteristics along which the initial data $g_0(z):\Sigma \subset {\mathbb R}^d\rightarrow {\mathbb R}$ ``propagates". This propagation embodies the method of characteristics.) In general, any initial data function $g_0(z)$  that is compatible with the PDE and with the initial set, $\Sigma$, generates a solution. Hence, the set of functions ${\mathbb G}_\lambda\subset C^1$ that solve Eq.~(\ref{qpde}) for each given $\lambda \in {\mathbb C}$ may be uncountable. 
Formalizing this argument requires pursuing sharp conditions concerning the domain of the transformation $h$, and the regularity of the vector field $F$ concerning existence and uniqueness of global solutions of the ODE and the KEIGs PDE, Eq.~(\ref{qpde}); all nontrivial issues that we did not tackle in this paper.  
{A formal treatment of this problem is handled by the concept of open eigenfunctions (see Definition 4.1 and Lemma 2 in \cite{mezic-2017}).}

This nontrivial geometric multiplicity in dimensions higher than one presents a major difficulty in matching, 
leading us to make the following, weakened matching statement, describing the homeomorphim in terms of two \textit{sets} of eigenfunctions, one set for each system. 
Similar to Cor.~(\ref{cor1}), the statement is local, and the domain of definition for the transformation $h$ will be extended afterwards.
\begin{cor}\label{cor2}
  Consider invariant subsets ${\mathbb X}, {\mathbb Y}\subset {\mathbb R^d}$, with   continuous time flows  $S_t^{(1)}$ and $S_t^{(2)}$ that are topologically orbit equivalent by $h:{\mathbb X}\rightarrow {\mathbb Y}$, and that each has eigenfunctions $g_{\lambda_i}^{(1)}(x)$, and $g_{\lambda_i}^{(2)}(y)$, with eigenvalues $\lambda_i$, $i=1,...,d$. The eigenfunctions are organized in two (column) vector-valued functions $\mathcal{G}^{(1)}$ and $\mathcal{G}^{(2)}$,
\begin{eqnarray}
{\cal G}^{(1)}(x)=[g_{\lambda_1}^{(1)}(x), g_{\lambda_2}^{(1)}(x), ..., g_{\lambda_d}^{(1)}(x)]^T, \nonumber \\
{\cal G}^{(2)}(y)=[g_{\lambda_1}^{(2)}(y), g_{\lambda_2}^{(2)}(y), ..., g_{\lambda_d}^{(2)}(y)]^T,
\end{eqnarray}
consisting of $d$ matching pairs of eigenfunctions, each pair associated to its own eigenvalue $\lambda_i=\lambda_i^{(1)}=\lambda_i^{(2)}$, and related by the same homeomorphism $h$ as in Eq.~(\ref{relate}). 
%
We assume that the %
function ${\cal G}^{(2)}:{\mathbb R}^d\rightarrow {\mathbb R}^d$ is invertible in an open set around a point $y\in {\mathbb Y}$. Equivalently, we can assume that the Jacobian of ${\cal G}^{(2)}$ is nonsingular at the point $y$, so that there is an open set $A\subset {\mathbb X}$ with $x\in {A}$ such that $y=h(x)$ is defined for all $x\in A$. In this case, we call $({\cal G}^{(1)},{\cal G}^{(2)})$ a {\bf complete set of eigenfunctions}, and together with the relation by $h$ we call them a {\bf componentwise matched complete set of eigenfunctions}.
 Furthermore, when ${\cal G}^{(2),-1}$ exists, the transformation between the spaces may be written,
\begin{equation}\label{thecor}
y=h(x)={\cal G}^{(2),-1}\circ {\cal G}^{(1)}(x).
\end{equation}                  
Again the degree of regularity follows that of the vector valued functions ${\cal G}^{(j)}$.  
\end{cor}

Note that as in Corollary \ref{cor1}, the implicit function theorem allows for existence and continuation of an inverse in a neighborhood of a point $y$ where the Jacobian is nonsingular, but this time not for a single eigenfunction, but rather the function ${\cal G}^{(2)}$, that is, the complete set of ``stacked" eigenfunctions of the second system.  Perhaps another useful standard geometric interpretation of such a function is that the level sets of each  $g_{\lambda_j}^{(2)}$ will be transverse to $g_{\lambda_i}^{(2)}$, in $A$ when $i\neq j$. 
The same is true for $g_{\lambda_j}^{(1)}$ if $h$ is a diffeomorphism in $A$. 
To discuss the domain of $h$, we may adapt the notation from Corollary \ref{cor1} (Eqs.~(\ref{range1})-(\ref{intersects}))  to state the following generalization of Eq.~(\ref{intersects}):
\begin{equation}\label{dr2}
D(h,{\mathbb X})=\cap_{i=1}^d D(g_{\lambda_i}^{(2),-1}\circ g_{\lambda_i}^{(1)},{\mathbb X})= \cap_{i=1}^d g^{(1),-1}_{\lambda_i}
(R_{2,i}\cap {\mathbb R^d}).
\end{equation}
Stating Corollary~\ref{cor2} in terms of a matched complete set of eigenfunctions is crucial, because for this set, {\it it is assumed that there exists a  homeomorphism $h$ that relates across all of them.}  For the non-multiplicity scenario, where there exists exactly one eigenfunction for each eigenvalue, this assumption would be automatically satisfied if we can find appropriate eigenvalue pairs.  The next example (3) demonstrates a scenario where a set of matched, two-dimensional eigenfunctions is known explicitly. After this example, we will discuss how allowing ourselves to constrain the function space can alleviate the nontrivial geometric multiplicity problem.

\subsubsection{Example 3:  Factorizing a Quadratic Transformation in 2D}

We now demonstrate the ``matching'' of a two-dimensional nonlinear system to a diagonal, linear system of the same dimension.
Finding the homeomorphism $h$ in this example also factorizes (diagonalizes) the nonlinear dynamical system.
Consider the following nonlinear example in a two dimensional domain.
Let $x(t)=(x_1(t),x_2(t))^T\in {\mathbb R}^2$ evolve according to,
\begin{eqnarray}\label{F1}
\dot{x}&=&F^{(1)}(x)=
\left(
\begin{array}{c}
-2 a_2 x_2(x_1^2-x_2-2x_1x_2^2+x_2^4)+a_1(x_1+4x_1^2x_2-x_2^2-8x_1x_2^3+4x_2^5) \\
2 a_1(x_1-x_2^2)^2-a_2(x_1^2-x_2-2x_1x_2^2+x_2^4 )
\end{array}
\right) \nonumber \\
&=&
\left(
\begin{array}{cc}
(x_1+4x_1^2x_2-x_2^2-8x_1x_2^3+4x_2^5) & -2  x_2(x_1^2-x_2-2x_1x_2^2+x_2^4) \\
2 (x_1-x_2^2)^2 & (x_1^2-x_2-2x_1x_2^2+x_2^4 )
\end{array}
\right)
\left(
\begin{array}{c}
a_1 \\ a_2 \end{array}\right)
\end{eqnarray}
The linear PDE, Eq.~(\ref{qpde}), becomes,
\begin{eqnarray}
\nabla g \cdot F^{(1)}(x)=&\frac{\partial g}{\partial x_1}(z)&[-2 a_2 x_2(x_1^2-x_2-2x_1x_2^2+x_2^4)+a_1(x_1+4x_1^2x_2-x_2^2-8x_1x_2^3+4x_2^5)]+ \nonumber \\
+ &\frac{\partial g}{\partial x_2}(z)&[2 a_1(x_1-x_2^2)^2-a_2(x_1^2-x_2-2x_1x_2^2+x_2^4 )]=\lambda g(x).
\end{eqnarray}
By substitution, we can confirm that there are  (at least two) solutions to this PDE,
\begin{eqnarray}
g_1^{(1)}(x)&=&(x_1-x_2^2) \label{KEIGSq1}\\
g_2^{(1)}(x)&=&(-x_1^2+x_2+2x_1x_2^2-x_2^4).\label{KEIGSq2}
\end{eqnarray}
By using the solution of the ODE, 
\begin{eqnarray}
x(t)&=&\left(\begin{array}{c}x_1(t)\\x_2(t)\end{array}\right)
\end{eqnarray}
where,
\begin{eqnarray}
x_1(t)=
e^{a_1 t}(x_{1,0}-x_{2,0}^2)+e^{2a_2t}(-x_{1,0}^2+x_{2,0}+2x_{1,0}x_{2,0}^2-x_{2,0}^4)+ \nonumber \\
 +2e^{(2a_1+a_2)t}(x_{1,0}-x_{2,0}^2)^2(-x_{1,0}^2+x_{2,0}+2x_{1,0}x_{2,0}^2-x_{2,0}^4)+ \nonumber \\
+ e^{4a_1 t}(x_{1,0}-x_{2,0}^2)^4, \nonumber \\
x_2(t)=e^{2a_1t}(x_{1,0}-x_{2,0}^2)^2+e^{a_2 t} (-x_{1,0}^2+x_{2,0}+2x_{1,0}x_{2,0}^2-x_{2,0}^4)
,
\end{eqnarray}
we can confirm by definition (\ref{eq:koopman definition}) that these are KEIGs.  In Fig.~\ref{figq}, we show the level sets of these KEIGs, Eq.~(\ref{KEIGSq1},\ref{KEIGSq2}).
It turns out, by our own construction, that they are in fact a matched complete set. We confirm this formally, by showing that they allow the construction of a transformation function $h$ matching this to the diagonal, linear system (\ref{lineardiag}).

\begin{figure}[hpt]     
	\centering   
        	\includegraphics[width=0.5\textwidth]{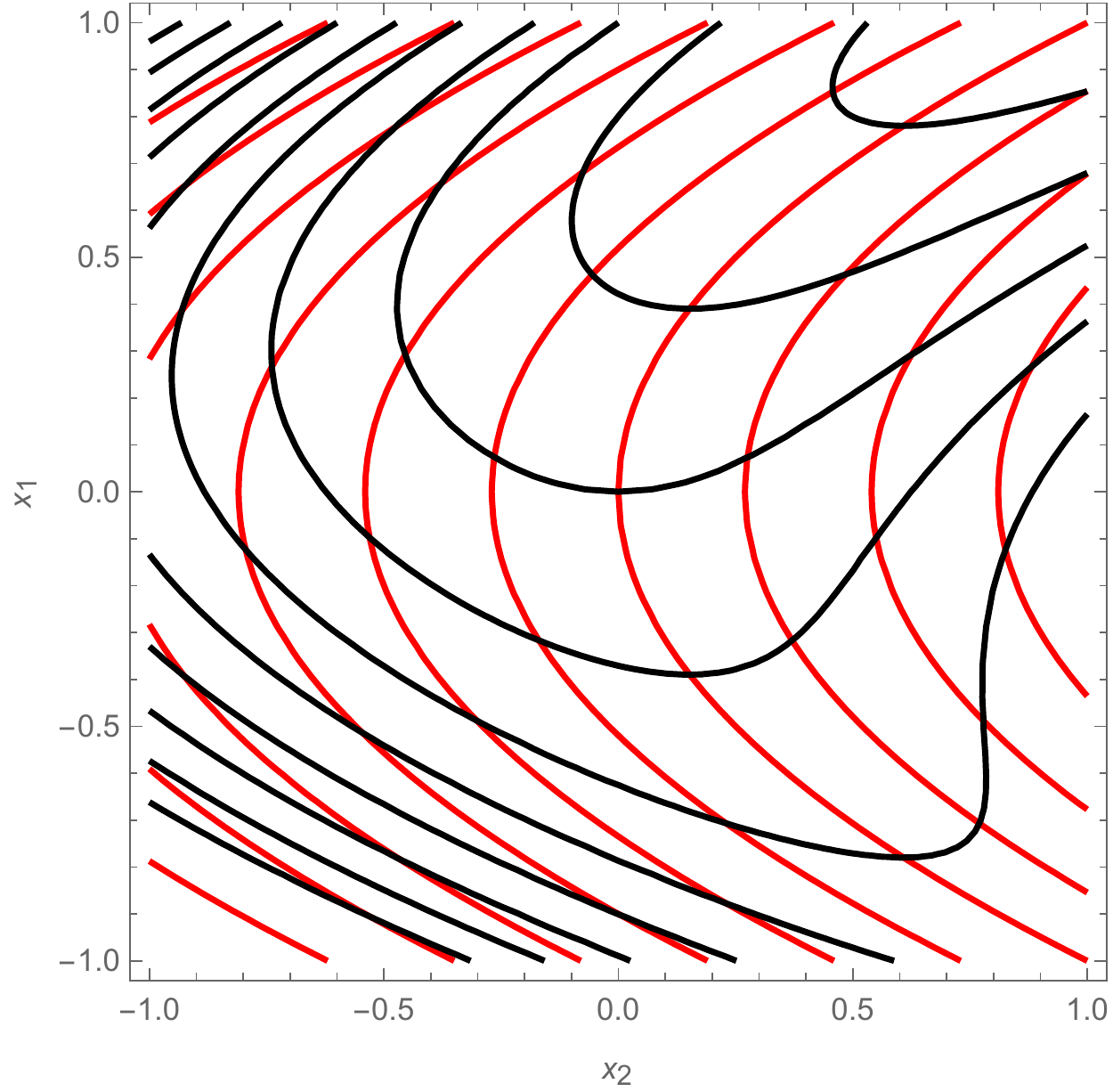}
       	\caption{ Level sets of the KEIGS $g_1(x)$ (red) and $g_2(x)$ (black) from Eq.~(\ref{KEIGSq1},\ref{KEIGSq2}).   }
	\label{figq}
\end{figure}

 Let $y(t)=[y_1(t),y_2(t)]^T$, and let 
\begin{equation}\label{lineardiag}
\dot{y}=A y=\mathrm{diag}(a_1,a_2) y.
\end{equation}
We know that this linear system has eigenfunctions that are state observers of each component, 
\begin{eqnarray}
g_1^{(2)}(y)&=&[y]_1=y_1\label{KEIGSqlin1} \\
g_2^{(2)}(y)&=&[y]_2=y_2.\label{KEIGSqlin2}
\end{eqnarray}
Defining the vector valued function, 
\begin{equation}\label{spf0}
{\cal G}^{(2)}(y)=[g_{\lambda_1}^{(2)}(y), g_{\lambda_2}^{(2)}(y)]^T=[y_1, y_2]^T,
\end{equation}
and using this as the matching complete set, Cor.~\ref{cor2} specializes to
\begin{eqnarray}\label{spf}
y&=&h(x)= \nonumber \\
&=&{\cal G}_{\lambda}^{(2),-1}\circ {\cal G}_{\lambda}^{(1)}(x)={\cal G}_{\lambda}^{(1)}(x)=[g_{\lambda_1}^{(1)}(x), g_{\lambda_2}^{(1)}(x)]^T= \nonumber \\
&=&[(x_1-x_2^2), (-x_1^2+x_2+2x_1x_2^2-x_2^4)]^T.
\end{eqnarray}                  
The inverse of the transformation follows,
\begin{equation}
x=h^{-1}(y)=[y_1+y_1^4+2y_1^2y_2+y_2^2, y_1^2+y_2]^T.
\end{equation}
With this transformation $h$, we confirm that 
\begin{equation}\label{transform}
\dot{y}=Dh(x)|_{x=h^{-1}(y)} \dot{x}_x=h^{-1}(y)=Dh(x)|_{x=h^{-1}(y)} F^{(1)}(x)_{x=h^{-1}(y)}.
\end{equation}
 Notice that written this way, $Dh$ is a Jacobian matrix since $h$ has a two-dimensional domain and range.
 It is straightforward to see that
\begin{equation}
Dh(x)=[\nabla g_{\lambda_1}^{(1)}(x), \nabla g_{\lambda_2}^{(1)}(x)]^T,
\end{equation}
is a matrix whose rows are gradients of the eigenfunctions of the $F_1$ system, Eq.~(\ref{F1}).
It then follows that
\begin{eqnarray}
\dot{y}&=&[\nabla g_{\lambda_1}^{(1)}(x): \nabla g_{\lambda_2}^{(1)}(y)]^T F_1(x) = \nonumber \\
&=& [\nabla g_{\lambda_1}^{(1)}(x)\cdot F^{(1)}(x), \nabla g_{\lambda_2}^{(1)}(x)\cdot F^{(1)}(x)]^T= \nonumber \\
&=& [ a_1 g_{\lambda_1}^{(1)}(x), a_2 g_{\lambda_2}^{(1)}(x) ]^T =\ \text{(transforming now to $y=h(x)$)} \nonumber \\
&=& [ a_1 g_{\lambda_1}^{(1)}(x), a_2 g_{\lambda_2}^{(1)}(x) ]^T |_{x=h^{-1}(y)} \nonumber \\
&=& [ a_1 g_{\lambda_1}^{(2)}(y), a_2 g_{\lambda_2}^{(2)}(y) ]^T \nonumber \\
&=& [a_1 y_1, a_2 y_2]^T=\mathrm{diag}(a_1,a_2) y.
\end{eqnarray}
The second step is, remarkably, a crisp restatement of the KEIG PDE (Thm.~\ref{thm2}) that appears naturally in the change of variables statement.  The step from the third to the fourth line follows the matching of the KEIGs described as state observers.
This line emphasizes the matching problem where we expand the description, 
\begin{equation}
[ a_1 g_{\lambda_1}^{(1)}(x), a_2 g_{\lambda_1}^{(1)}(x) ]^T |_{x=h^{-1}(y)}=[a_1 y_1, a_2 y_2]^T.
\end{equation}
This crucially relies on  the assumption 
that the eigenfunctions are a matched complete set, all related by the same homeomorphism. We know that this can fail to be true even if the eigenvalues are equal (as illustrated in Appendix \ref{defectexample}).

At this point, we confess that in the above example we knew that we had a matched complete pair, because the nonlinear system was constructed through a transformation of the linear system in the first place; we then ``threw away'' the transformation, so that we could (re)discover it through Cor.~\ref{cor2}.
%
%
  In Section \ref{defect}, we show that if we sufficiently restrict the function space over which the solutions of the KEIGs PDE exist, we can guarantee that the associated Koopman operator no longer has nontrivial geometric multiplicity. In the appropriately restricted context, we can recover the same homeomorphism $h$ for the example discussed here, without the additional assumption of a matched complete set. The overall approach is generically applicable, and even has a numerical extension, which we describe in Sec.~(\ref{sec:num example}).

%
%

\subsubsection{Example 4:  Rectifying the Linear Problem in 2D}

We now demonstrate the rectification of the diagonal linear problem in two dimensions (here: on $\mathbb{R}^+\times\mathbb{R}^+$; other quadrants will also work.  See below). Obviously, if successful, we can then rectify the nonlinear problem of the previous section through the composition of the two transformations (in the appropriate domain).
We will use the KEIGs to transform the linear diagonal system, Eq.~(\ref{lineardiag2}) into the rectified system (Eq.~\ref{rect2}). The linear system is again
\begin{equation}\label{lineardiag2}
\dot{x}=F^{(1)}(x)=\mathrm{diag}(a_1,a_2) x,
\end{equation}
with $S^{(1)}_t(x_0)=[x_{0,1}e^{a_1 t}, x_{0,2}e^{a_2 t}]^T$. The vector valued functions are (see Eq.~(\ref{KEIGSqlin1},\ref{KEIGSqlin2})), 
\begin{equation}\label{spf0}
G^{(1)}(y)=[g_{\lambda_1}^{(1)}(x), g_{\lambda_2}^{(1)}(y)]^T=[x_1, x_2]^T.
\end{equation}
To rectify the system, we choose the following system as the target to match to:
\begin{equation}\label{rect2}
\dot{y}=F^{(2)}(y)=[1, 0]^T,
\end{equation}
so that the flow is, $S^{(2)}_t(y_0)=[y_{0,1}+t, y_{0,2}]^T$. The corresponding KEIGs PDE becomes,
\begin{equation}
\nabla g \cdot F_2(y)=\frac{\partial g}{\partial y_1} 1 + \frac{\partial g}{\partial y_2} 0=\lambda g,
\end{equation}
which has a solution,
\begin{equation}
g_\lambda^{(2)}(y)=e^{\lambda^{(2)} y_1+q(y_2)},
\end{equation}
where, $q$ is an arbitrary $C^1$ function, which can be considered as representing initial data on an initial set that is transverse to the flow (see the discussion at the beginning of Sec.~\ref{defectsection}). Specifically, 
 this initial set could be taken to be the $y_2$-axis, $S=\{(y_1,y_2): y_1=0\}$, which is clearly transverse to the flow, and $q(y_2)=g_0(0,y_2)$ represents the initial data. This choice is consistent with the PDE, which has trivial characteristics. 
We define the vector valued KEIGs, with $y\in {\mathbb R}^2$, to be
\begin{equation}\label{g2i1}
{\cal G}^{(2)}_\lambda(y)=[e^{\lambda_1^{(2)} y_1}e^{q(y_2)}, e^{\lambda_2^{(2)} y_1}e^{q(y_2)}]^T,
\end{equation}
where the vector of eigenvalues $\lambda=[\lambda_1^{(2)},\lambda_2^{(2)}]^T$ must be chosen so that $\lambda_1^{(2)}\neq \lambda_2^{(2)}$, to ensure a complete set of KEIGs.
The inverse follows, in $\mathbb{R}^+\times\mathbb{R}^+$ (for positive $z_1$ and $z_2$),
\begin{equation}\label{g2i2}
y={\cal G}^{(2),-1}_\lambda(z)=\left[\frac{\log z_1-\log z_2}{\lambda_1^{(2)}-\lambda_2^{(2)}}, q^{-1}\left(\frac{\lambda_1^{(2)} \log z_2 -\lambda_2^{(2)} \log z_1}{\lambda_1^{(2)}-\lambda_2^{(2)}}\right)\right]^T,
\end{equation}
when the function $q$ is invertible.  For simplicity, choose $q$ to be the identity function. It is important to note that different, invertible functions $q$ would give different transformations, and thus different rectifications (see Fig.~\ref{fig:different initial functions}). 
\begin{figure}[ht]
\centering
\includegraphics[width=0.9\textwidth]{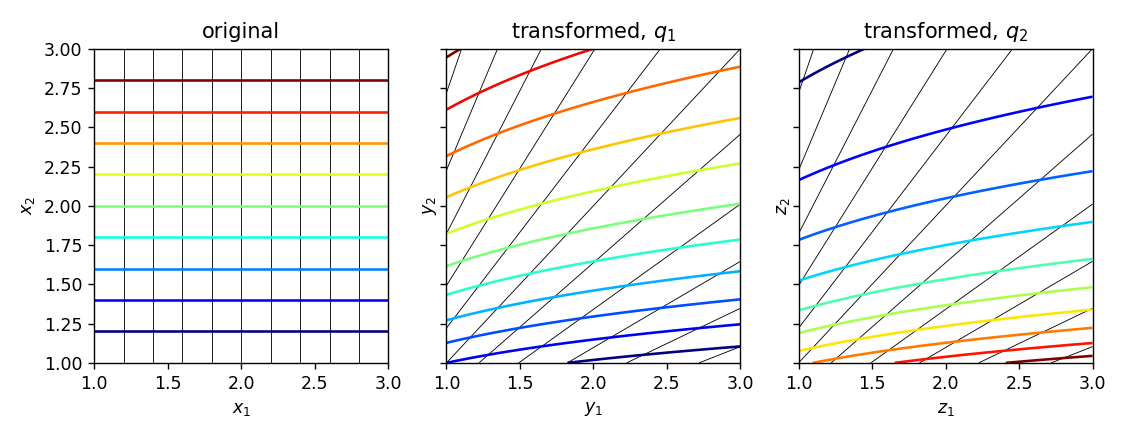}
\caption{\label{fig:different initial functions}Choosing different initial values on the initial set leads to different homeomorphisms $h$. The plots show the original coordinates (left) with $x_1$ level sets as thin black lines, and $x_2$ level sets as colored lines. The coordinates resulting from the homeomorphisms $h$ (Eq.~\ref{g2i2}) are shown in the middle and right plot, with $q_1^{-1}=x$ and $q_2^{-1}=\ln(10+\exp(-x))$. Notice that changing $q$ results in a system with the same differential equations, but different coordinates. }
\end{figure}
Then following Corollary \ref{cor2}, Eq.~(\ref{thecor}),
\begin{equation}
y=h(x)={\cal G}^{(2),-1}\circ {\cal G}^{(1)}(x)=\left[\frac{\log x_1-\log x_2}{\lambda_1^{(2)}-\lambda_2^{(2)}}, \frac{\lambda_1^{(2)} \log x_2 -\lambda_2^{(2)} \log x_1}{\lambda_1^{(2)}-\lambda_2^{(2)}}\right]^T,
\end{equation}
for $x\in(\mathbb{R}^+\times \mathbb{R}^+)=:D$, the domain of $h$. 
%
 As a check that this change of variables rectifies the problem, when the Jacobian derivative $Dh|_x$ exists at points $x=h^{-1}(y)$, we see that
\begin{eqnarray}
\dot{y}&=&Dh|_x \cdot \dot{x}=Dh|_{x=h^{-1}(y)} \cdot F^{(1)}(x)|_{x=h^{-1}(y)}=
\frac{1}{\lambda_1^{(2)}-\lambda_2^{(2)}} \left(
\begin{array}{cc}
\frac{1}{x_1} & \frac{-1}{x_2} \\
\frac{-\lambda_2^{(2)}}{x_1} & \frac{\lambda_1^{(2)}}{x_2}
\end{array}
\right)
\left(
\begin{array}{c}
a_1 x_1 \\ a_2 x_2 \end{array}
\right)=  \nonumber \\
&=&      
\frac{1}{\lambda_1^{(2)}-\lambda_2^{(2)}}\left(
\begin{array}{c}
a_1-a_2 \\ \lambda_1^{(2)} a_2-\lambda_2^{(2)} a_1
\end{array}
\right)= \nonumber \\
&=& 
\left(
 \begin{array}{c}
1 \\ 0
\end{array}\right), \mbox{ if }\lambda_1^{(2)}=a_1, \mbox{ and, }\lambda_2^{(2)}=a_2, \mbox{ but only if, }a_1\neq a_2.
\end{eqnarray}
Thus the linear system can be rectified, if the eigenvalues $\lambda_1^{(2)}$ and $\lambda_2^{(2)}$ are chosen to match the eigenvalues of the Jacobian of the linear system, at least in the spatial domain where all of the transformations exist and are continuously differentiable. In this example, the domain of $h$ is the first quadrant of $\mathbb{R}^2$, considering the logarithms involved in Eq.~(\ref{g2i2}). Other quadrants are also possible, depending on the choice of $\mathcal{G}^{(2)}$. Since in Example~(3) we showed that the nonlinear example can be linearized, and thus by composition this  nonlinear example can be rectified on the domain $D$ of $h$. 
See Fig.~\ref{fig:composition of transforms}.
\begin{figure}[ht!]
\centering
\includegraphics[width=0.5\textwidth]{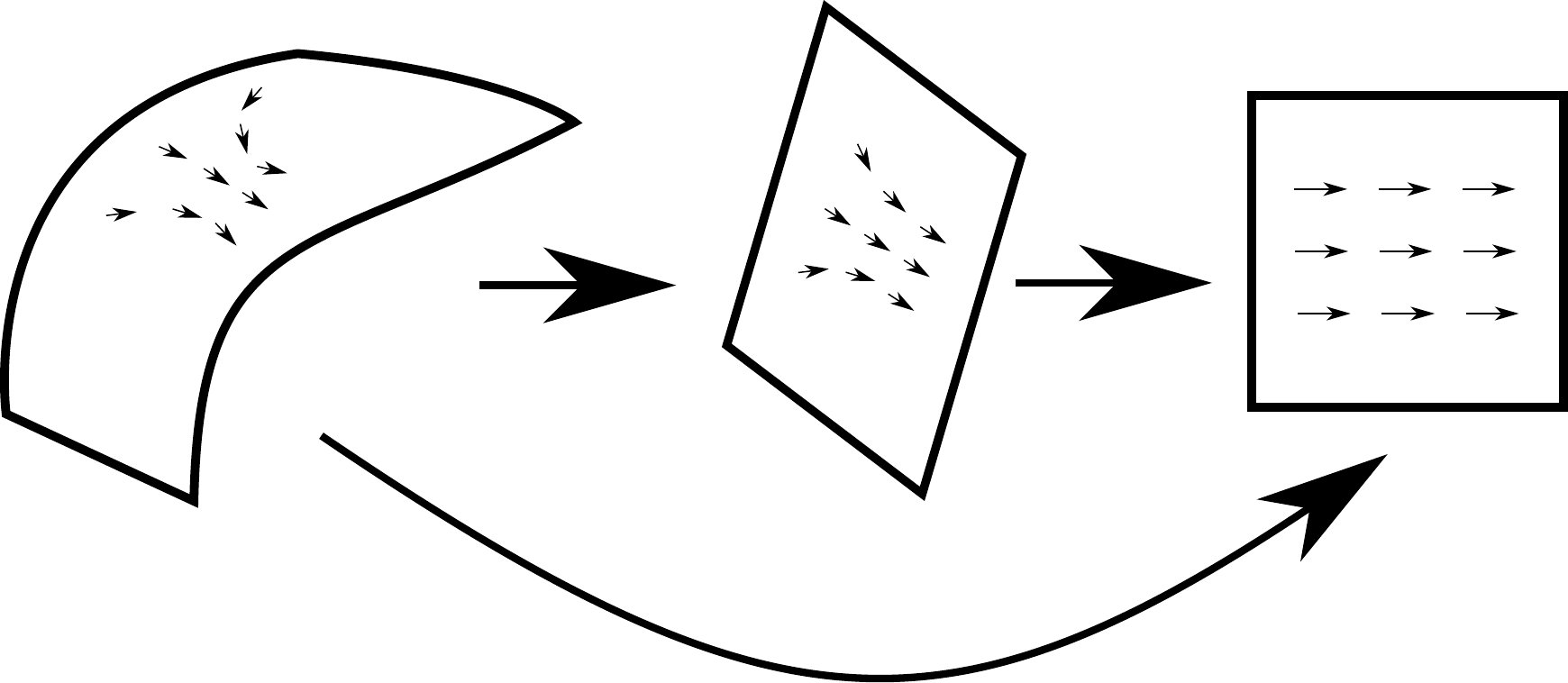}
\caption{\label{fig:composition of transforms}Schematic representation of the rectification of a nonlinear system (left), either by matching to a linear system first (middle), and rectifying the linear system (right), or directly matching the nonlinear to the rectified system.}
\end{figure}
Indeed, matching to a reference system provides a systematic way of matching two arbitrary systems. Here, a simple system such as a linear system or a ``flow box'' system (meaning a system with constant vector field) is sufficient.  It reminds the authors of matching between two dynamical systems by matching each one of them to the appropriate normal form.

\section{An Approach to Circumvent Nontrivial Geometric Multiplicity}\label{defect}

The previous section discussed the problem of nontrivial geometric multiplicity: 
Building complete sets of eigenfunctions does not necessarily imply that they are matched (that each pair is related by the same homeomorphism).
Consider, for example, the case of the identity transformation $h$ to a perfect copy of a system. 
By the nontrivial geometric multiplicity, 
a complete set of eigenfunctions $G^{(1)}(x)$ may be developed for system $F^{(1)}$, and likewise  $G^{(2)}(x)$ for system $F^{(2)}=F^{(1)}$,  but $G^{(1)}\neq G^{(2)}$. If a solution to the equation 
$h(x)=G^{(2),-1}\circ G^{(1)}(x)$ exists, it may not be the identity function, $h(x)=x$ (see Appendix \ref{defectexample} for an explicit example).  Clearly, \textbf{nontrivial geometric multiplicity is a major obstacle for systematic spectral matching of dynamical systems} through KEIGs. 
If there is a solution to the spectral matching problem, then it is necessary to have a way to systematically select  representatives $g_{\lambda_i} \in {\mathbb G}_{\lambda_i}$ in a consistent way, leading to a complete \textit{matched} set (see Cor.~\ref{cor2}), where ${\mathbb G}_{\lambda_i}$ is the set of all admissible eigenfunctions for a given eigenvalue $\lambda_i$. 
We will show here that this can be circumvented by sufficiently restricting the function space. 


\subsection{The EDMD-M framework}

The approach we propose bears conceptual and algorithmic similarities with a computational algorithm for approximating Koopman eigenfunctions known as EDMD, Extended Dynamic Mode Decomposition\cite{williams-2015,williams-2015b,li-2017}. For that reason, we will call our approach EDMD-M, where ``M'' stands for Matching.
An important step in EDMD is the (computationally practical) choice of a finite dictionary. In our case, this choice does not conceptually have to be finite, although it will become so in certain computational implementations.
The crucial additional step in EDMD-M is the selection of \textit{the same} dictionary for the two systems to be matched, and the assumption that the matching function $h$ lies (possibly approximately) in the span of this dictionary.

Consider two dynamical systems,
\begin{eqnarray}
\dot{z}^{(1)} & = & F^{(1)}(z^{(1)}),\label{eq:sys_1}\\
\dot{z}^{(2)} & = & F^{(2)}(z^{(2)}).\label{eq:sys_2}
\end{eqnarray}
both in $\mathbb{R}^{d}$, for which there exists
some $h:\mathbb{R}^{d}\rightarrow\mathbb{R}^{d}$ such that the two systems are topologically orbit equivalent, through 
$z^{(2)}=h(z^{(1)}).$
%
%
One way to circumvent the problem of nontrivial geometric multiplicity is to
consider finite-dimensional projections of the two corresponding Koopman operators.
Let $\mathcal{K}_{F^{(i)}}$ denote the two Koopman operators arising from the corresponding two infinitesimal generators $A_{\mathcal{K}}^{(i)}$, $i=1,2$. In addition, denote by $\mathcal{C}_h$ the composition operator associated with the map $h$, i.e.
\begin{align}
\mathcal{K}_{F^{(i)}} \phi &:= \phi \circ A_{\mathcal{K}}^{(i)} \\
\mathcal{C}_h \phi &:= \phi \circ h.
\end{align}
Fixing $N$, consider a set of functions $\Psi:=(\psi_{1},\psi_{2},\dots,\psi_{N})$, $\psi_{j}\in \mathcal{F}$ and denote their span
\begin{equation}
U:=\mathrm{Span}(\Psi) = \left\{\sum_{j=1}^{N}a_j\psi_j,a\in\mathbb{C}^N\right\}.
\end{equation}
We make the following crucial assumptions

\begin{enumerate}
	\item[(A1)] $\mathcal{K}_{F^{(i)}}U \subseteq U$ for $i=1,2$. 
	\item[(A2)] $\mathcal{C}_{h}U \subseteq U$.
	\item[(A3)] $\{\psi_1,\dots,\psi_N\}$ are linearly independent.
\end{enumerate}Here, (A1) states that the projection of the Koopman dynamics on $U$
is closed and (A2) states that the composition mapping by $h$ is also
closed on $U$. Let us denote the restrictions (to $U$) of $\mathcal{K}_{F^{(i))}}$
by $K^{(i)}\in\mathbb{C}^{N\times N}$ and $\mathcal{C}_{h}$
by $H\in\mathbb{C}^{N\times N}$ i.e., for any $\phi\in U$ with 
\begin{equation}
\phi=a^{T}\Psi=\sum_{j=1}^{N}a_{j}\psi_{j},
\end{equation}
we have 
\begin{equation}
\mathcal{K}_{F^{(i)}}\phi=a^{T}K^{(i)}\Psi\qquad\mathcal{C}_{h}\phi=a^{T}H\Psi.\label{eq:proj_def}
\end{equation}
From ``usual EDMD'' analysis\cite{williams-2015}, it is clear that if $v_{\lambda}^{(i)}$
is a left eigenvector of $K^{(i)}$ with eigenvalue $\lambda$,
then
\begin{equation}
g_{\lambda}^{(i)}=v_{\lambda}^{(i),T}\Psi\in U
\end{equation}
is an eigenfunction of $\mathcal{K}_{F^{(i)}}$ with the
same eigenvalue $\lambda$. Let us prove a converse statement:
\begin{lem}
\label{lem:lemma1}Let $i\in\left\{ 1,2\right\} $, $\lambda\in\mathbb{C}$
and $g_{\lambda}\in U$ such
that $\mathcal{K}_{F^{(i)}}g_{\lambda}=\lambda g_{\lambda}$.
Suppose assumptions (A1) and (A3) hold. Then, there exists a unique
$v_{\lambda}\in\mathbb{C}^{N}$ such that $g_{\lambda}=v_{\lambda}^{T}\Psi$
and $v_{\lambda}^{T}K^{(i)}=\lambda$v$_{\lambda}^{T}$. \end{lem}
\begin{proof}
Since $g_{\lambda}^{(i)}\in U$, we write 
\begin{equation}
g_{\lambda}=v_{\lambda}^{T}\Psi
\end{equation}
for some $v_{\lambda}\in\mathbb{C}^{N}$, which is unique since $\left\{ \psi_{j}\right\} $
are linearly independent (A3). Assuming (A1), using the definition
of $K^{(i)}$ (\ref{eq:proj_def}) we have for each $z$
\begin{equation}
\lambda v_{\lambda}^{T}\Psi(z)=\lambda g_{\lambda}(z)=\mathcal{K}_{F^{(i)}}g_{\lambda}(z)=v_{\lambda}^{T}K^{(i)}\Psi(z).
\end{equation}
Hence, 
\begin{equation}
\left[\lambda v_{\lambda}^{T}-v_{\lambda}^{T}K^{(i)}\right]\Psi(z)=0
\end{equation}
for all $z\in\mathbb{R}^{d}$. By linear independence (A3) we have
\begin{equation}
\lambda v_{\lambda}^{T}=v_{\lambda}^{T}K^{(i)}.
\end{equation}
\end{proof}
With lemma~(\ref{lem:lemma1}), we can prove the following theorem that guarantees
the reconstruction of the function $h$ using \textit{separate} spectral analysis
of systems (1) and (2), admittedly under restrictive conditions. In particular,
we make the additional assumptions:

\begin{enumerate}
	\item[(A4)] The coordinate-projection maps $P_k(z)=[z]_k=z_k$, $k=1,\dots,d$ belong to $U$, i.e. there exists $B\in\mathbb{R}^{d\times N}$ s.t. $P=B\Psi$.
	\item[(A5)] We have access to a pair of matching values $z^{(1)}_0,z^{(2)}_0=h(z^{(1)}_0)\in\mathbb{R}^d$, such that 
	\begin{equation}
	\sum_{k=1}^N V_{jk}^{(1)}\psi_k(z^{(2)}_0)\neq0
	\end{equation} for all $j=1,\dots,N$, where $V^{(i)}$ a matrix whose $j^{th}$ row is a left eigenvector of $K^{(i)}$.
\end{enumerate}
The matching pair required in (A5) fixes the constant factors of the eigenfunctions, and this selects single elements from each equivalence classes (as also noted in\cite{williams-2015b}). This selection can be interpreted as a \textit{pinning condition} for symmetries that might be present in the systems.
\begin{thm}\label{thm:edmd-m}
Consider a dictionary $\Psi$ for which assumptions (A1)-(A3) are satisfied.
Suppose further that for each $i=1,2$, $K^{(i)}$ is diagonalizable with distinct eigenvalues. Then, 
\begin{enumerate}
	\item[(i)] $K^{(1)}$ and $K^{(2)}$ are similar.
\end{enumerate}
Assume further (A4), (A5) are satisfied, then
\begin{enumerate}
	\item[(ii)] $h(z)=B\left[V^{(2)}\right]^{-1}DV^{(1)}\Psi(z)$. 
\end{enumerate}
where $V^{(i)}$'s rows are eigenvectors with matched eigenvalues and $D$ is a diagonal matrix with entries
\begin{equation}
D_{jj}=\frac{\sum_{k=1}^N V_{jk}^{(2)}\psi_k(z^{(2)}_0)}{\sum_{k=1}^N V_{jk}^{(1)}\psi_k(z^{(1)}_0)}.
\end{equation}
\label{thm:edmd_m}
\end{thm}
\begin{proof}
Let $\left\{ \lambda_{1}^{(2)},\dots,\lambda_{N}^{(2)}\right\} $
be the $N$ distinct eigenvalues of $K^{(2)}$ with corresponding
left eigenvectors $\left\{ v_{1}^{(2)},\dots,v_{N}^{(2)}\right\} $.
For each $j$, we know that 
\begin{equation}
g_{j}^{(2)}:=v_{j}^{(2),T}\Psi\in U
\end{equation}
is an eigenfunction of $\mathcal{K}_{F^{(2)}}$ with eigenvalue
$\lambda_{j}^{(2)}$. By conjugacy, we know that 
\begin{equation}
g_{j}^{(1)}:=g_{j}^{(2)}\circ h
\end{equation}
is an eigenfunction of $\mathcal{K}_{F^{(1)}}$ with the
same eigenvalue $\lambda_{j}^{(2)}$. From assumption (A2),
we know that $g_{j}^{(1)}\in U$ and hence by Lemma \ref{lem:lemma1},
there exists a unique $v_{j}^{(1)}$ such that 
\begin{equation}
g_{j}^{(1)}=v_{j}^{(1),T}\Psi\qquad v_{j}^{(1),T}K^{(1)}=\lambda_{j}^{(2)}v_{j}^{(1),T}.
\end{equation}
In particular, $v_{j}^{(1)}$ is a left eigenvector of
$K^{(1)}$ with eigenvalue $\lambda_{j}^{(2)}$.
Since the eigenvalues of $K^{(1)}$ are distinct, they
must also be $\left\{ \lambda_{1}^{(2)},\dots,\lambda_{N}^{(2)}\right\} .$
Hence, $K^{(1)}$and $K^{(2)}$ are similar
and this proves (i). 

Next, observe that 
\begin{equation}
\mathcal{K}_{h}(V^{(2)}\Psi)=(V^{(2)}\Psi)\circ h=\left(\begin{array}{c}
g_{1}^{(2),T}\circ h\\
\vdots\\
g_{N}^{(2),T}\circ h
\end{array}\right)=\left(\begin{array}{c}
g_{1}^{(1),T}\\
\vdots\\
g_{N}^{(1),T}
\end{array}\right)=DV^{(1)}\Psi,
\end{equation}
where $D$ is a diagonal matrix yet undetermined, because eigenvectors are only unique up to a multiplicative constant. Now, for each $j$,
we have
\begin{equation}
D_{jj}v_{j}^{(1),T}\Psi(z)=v_{j}^{(2),T}\Psi\circ h(z)
\end{equation}
for all $z\in\mathbb{R}^{d}$. Using (A5), we can determine $D$ as 
\begin{equation}
D_{jj}=\frac{\sum_{k=1}^N V_{jk}^{(2)}\psi_k(z^{(2)}_0)}{\sum_{k=1}^N V_{jk}^{(1)}\psi_k(z^{(1)}_0)}.
\end{equation}
Since $\mathcal{C}_{h}(V^{(2)}\Psi)=DV^{(1)}\Psi$,
using the definition (\ref{eq:proj_def}) of $H$, we get 
\begin{equation}
V^{(2)}H\Psi=DV^{(1)}\Psi\implies H=\left[V^{(2)}\right]^{-1}DV^{(1)}.
\end{equation}
Since the projection map $P\in U$ with $P=B\Psi$, we have 
\begin{equation}
h=P\circ h=\mathcal{K}_{h}P=BH\Psi=B\left[V^{(2)}\right]^{-1}DV^{(1)}\Psi.
\end{equation}
\end{proof}


\subsection{Example 5: Matching two 2D linear systems}
Consider two 2D linear dynamical systems 
\begin{equation}
F^{(1)}(x) = \left(\begin{array}{c}
x_{2}\\
x_{1}
\end{array}\right), \qquad
F^{(2)}(y) = \left(\begin{array}{c}
y_{1}\\
-y_{2}
\end{array}\right).
\end{equation}
These are matching systems with the transformation
\begin{equation}
y = h(x) = \left(\begin{array}{c}
x_{1}+x_{2}\\
x_{1}-x_{2}
\end{array}\right).
\end{equation}
We aim to discover this transformation from separate Koopman analysis of the two systems. We take as a dictionary
\begin{equation}
\Psi = (z_1, z_2)
\end{equation}
which satisfies conditions (A1)-(A4) with $B=I$. We denote $z_i$ both for $x_i$ and $y_i$, as the dictionary is the same for both systems. The projections of the two Koopman operators onto this subspace are easily obtained by following the actions of the generators $F^{(i)}\cdot\nabla$ on the functions $\psi_1(z)=z_1,\psi_2(z)=z_2$. We have
\begin{equation}
K^{(1)} = \left(\begin{array}{cc}
0 & 1\\
1 & 0
\end{array}\right)\qquad
K^{(2)} = \left(\begin{array}{cc}
1 & 0\\
0 & -1
\end{array}\right).
\end{equation}
Clearly, $K^{(1)}$ and $K^{(2)}$ are similar with eigenvalues $(-1,+1)$ and eigenvector matrices
\begin{equation}
V^{(1)} = \left(\begin{array}{cc}
-1 & 1\\
1 & 1
\end{array}\right)\qquad
V^{(2)} = \left(\begin{array}{cc}
0 & 1\\
1 & 0
\end{array}\right).
\end{equation}
Now assume that $x_0=(1,2)$ and $y_0=h(x_0)=(3,-1)$ are known, hence $D = \mathrm{diag}(1,-1)$. Using formula (ii) in Theorem \ref{thm:edmd_m}, we get
\begin{equation}
h(x) = B\left[V^{(2)}\right]^{-1}D V^{(1)} \Psi(x) = \left(\begin{array}{cc}
1 & 1\\
1 & -1
\end{array}\right)\Psi(x) = \left(\begin{array}{c}
x_1+x_2\\
x_1-x_2
\end{array}\right).
\end{equation}

\subsection{Example 6: Quadratic Example in 2D using a finite-dimensional projection}
In the previous example, we easily found a finite dimensional subspace $U=\mathrm{Span}(\Psi)$ on which various operations are closed. This is due to the linear structure of the problem. For general nonlinear systems or transformations,  such a finite-dimensional subspace may not (and in general, will not) exist. However, we show in the example that follows that formally, this approach also works in the countably infinite-dimensional setting. 
We revisit the problem of Example 3 and apply the EDMD-M method. 
\begin{equation}
F^{(1)}(x)=\left(\begin{array}{c}
a_{1}(x_{1}-x_{2}^{2})(1+4x_{1}x_{2}-4x_{2}^{3})-2a_{2}x_{2}((x_{1}-x_{2}^{2})^{2}-x_{2})\\
2a_{1}(x_{1}-x_{2}^{2})^{2}-a_{2}((x_{1}-x_{2}^{2})^{2}-x_{2})
\end{array}\right)
\end{equation}
and 
\begin{equation}
F^{(2)}(y)=\left(\begin{array}{c}
a_{1}y_{1}\\
a_{2}y_{2}
\end{array}\right).
\end{equation}
The goal is to find 
\begin{equation}
y=h(x_{1},x_{2})=\left(\begin{array}{c}
x_{1}-x_{2}^{2}\\
-x_{1}^{2}+x_{2}+2x_{1}x_{2}^{2}-x_{2}^{4}
\end{array}\right)
\end{equation}
using Theorem \ref{thm:edmd_m}. Let us consider the set of functions 
\begin{equation}
\left\{ \tilde{\psi}_{m,n}=z_{1}^{m}z_{2}^{n}:m,n\geq0,mn\neq0\right\} .
\end{equation}
The form of $\tilde{\psi}_{m,n}$ is reminiscent of the Carleman linearization\cite{tsiligiannis-1987,carleman-1932,banks-1992,kowalski-1991}, which has been useful lately for system identification by least-squares minimization, l1 minimization (sparse/compressed-sensing type), \cite{brunton2016discovering,napoletani2008reconstructing,wang2016data,yao2007modeling} or conjugacy defect minimization \cite{skufca2007relaxing}. In fact, the matrices $K$ approximating the Koopman operators here play the same role as the Carleman matrices. However, following the ideas from the original EDMD\cite{williams-2015,williams-2015b}, we could have chosen any suitable dictionary, not just multinomials. In fact, in the numerical example described in Sec.~\ref{sec:num example}, we will use a neural net to find the dictionary from data without prescribing it.

To improve indexing, we can use the bijective Cantor-Pairing function
\begin{equation}
c:\mathbb{N}\times\mathbb{N}\rightarrow\mathbb{N}
\end{equation}
with 
\begin{equation}
c(m,n)=\frac{1}{2}(m+n)(m+n+1)+n.
\end{equation}
Then, we define 
\begin{equation}
\psi_{j}=\tilde{\psi}_{c^{-1}(j)},\qquad\Psi=(\psi_{1},\psi_{2},\dots)\qquad U=\text{Span}\left\{ \Psi\right\} .
\end{equation}
For concreteness, the first 9 components are 
\begin{equation}
\Psi(z)=\left(\begin{array}{c}
z_{1}\\
z_{2}\\
z_{1}^{2}\\
z_{1}z_{2}\\
z_{2}^{2}\\
z_{1}^{3}\\
z_{1}^{2}z_{2}\\
z_{1}z_{2}^{2}\\
z_{2}^{3}\\
\vdots
\end{array}\right).
\end{equation}
With this choice, it is clear that (A1)-(A4) are satisfied.
For (A5),
we pick $z^{(1)}_{0}=(2,2)$ so that $z^{(2)}_{0}=h(z^{(1)}_{0})=(-2,-2)$ and assume that we have this pair of points. 

With the current indexing one can check that the projections of $\mathcal{K}_{F^{(i)}}$
to $U$ is very simple: they are given by
\begin{equation}
K^{(1)}=\left(\begin{array}{ccccccccccc}
a_{1} & 0 & 0 & 0 & -a_{1}+2a_{2} & 0 & 4a_{1}-2a_{2} & 0 & 0\\
0 & a_{2} & 2a_{1}-a_{2} & 0 & 0 & 0 & 0 & -4a_{1}+2a_{2} & 0\\
0 & 0 & 2a_{1} & 0 & 0 & 0 & 0 & -2a_{1}+4a_{2} & 0\\
0 & 0 & 0 & a_{1}+a_{2} & 0 & 2a_{1}-a_{2} & 0 & 0 & -a_{1}+2a_{2}\\
0 & 0 & 0 & 0 & 2a_{2} & 0 & 4a_{1}-2a_{2} & 0 & 0\\
0 & 0 & 0 & 0 & 0 & 3a_{1} & 0 & 0 & 0 & \cdots & \cdots\\
0 & 0 & 0 & 0 & 0 & 0 & 2a_{1}+a_{2} & 0 & 0\\
0 & 0 & 0 & 0 & 0 & 0 & 0 & a_{1}+2a_{2} & 0\\
0 & 0 & 0 & 0 & 0 & 0 & 0 & 0 & 3a_{2}\\
&  &  &  &  & \vdots &  &  &  & \ddots\\
&  &  &  &  & \vdots &  &  &  &  & \ddots
\end{array}\right)
\end{equation}
and
\begin{equation}
K^{(2)}=\left(\begin{array}{ccccccccccc}
a_{1} & 0 & 0 & 0 & 0 & 0 & 0 & 0 & 0\\
0 & a_{2} & 0 & 0 & 0 & 0 & 0 & 0 & 0\\
0 & 0 & 2a_{1} & 0 & 0 & 0 & 0 & 0 & 0\\
0 & 0 & 0 & a_{1}+a_{2} & 0 & 0 & 0 & 0 & 0\\
0 & 0 & 0 & 0 & 2a_{2} & 0 & 0 & 0 & 0\\
0 & 0 & 0 & 0 & 0 & 3a_{1} & 0 & 0 & 0 & \cdots & \cdots\\
0 & 0 & 0 & 0 & 0 & 0 & 2a_{1}+a_{2} & 0 & 0\\
0 & 0 & 0 & 0 & 0 & 0 & 0 & a_{1}+2a_{2} & 0\\
0 & 0 & 0 & 0 & 0 & 0 & 0 & 0 & 3a_{2}\\
&  &  &  &  & \vdots &  &  &  & \ddots\\
&  &  &  &  & \vdots &  &  &  &  & \ddots
\end{array}\right)
\end{equation}
Again, these matrices are obtained by following the action of the generators $F^{(i)}\cdot\nabla$
on the functions in $\Psi$. Notice that both $K^{(1)}$
and $K^{(2)}$ have distinct (and identical) eigenvalues
\begin{equation}
\lambda_{j}=c^{-1}(j)^{T}\left(\begin{array}{c}
a_{1}\\
a_{2}
\end{array}\right).
\end{equation}
We can calculate their eigenvector matrices (which we will not write down for brevity). The projection operators are also included, with
$P=B\Psi$ and
\begin{equation}
B=\left(\begin{array}{ccccc}
1 & 0 & 0 & 0 & \cdots\\
0 & 1 & 0 & 0 & \cdots
\end{array}\right).
\end{equation}
We shall take the first eigenvectors' first 14 dimensions (enough to span $h$). With this choice,
we get 
\begin{equation}
D=\mathrm{Diag}\left\{ -1,1,-\frac{5}{3},1,-1,\frac{17}{4},-\frac{5}{3},1,3,1,-\frac{5}{2},1,-5,1\right\}.
\end{equation}
Using the formula
\begin{equation}
h=B\left[V^{(2)}\right]^{-1}DV^{(1)}\Psi
\end{equation}
we have indeed
\begin{equation}
h(x)=\left(\begin{array}{c}
x_{1}-x_{2}^{2}\\
-x_{1}^{2}+x_{2}+2x_{1}x_{2}^{2}-x_{2}^{4}
\end{array}\right).
\end{equation}
Notice that this result was already arrived at by other means in Eq.~(\ref{spf}).

\subsection{Numerical Implementations of EDMD-M}


In the above examples, we treated EDMD-M as a method to find the matching transformation symbolically. 
We now briefly discuss a possible numerical approach. 
The key is to find a dictionary set $\Psi$ that has the desired properties (A1)-(A4). 
Note that (A4) is easy to satisfy by simply including the projection maps into the dictionary set. 
The rest of the dictionary elements must be selected to be broad enough so that (A1) and (A2)
can be satisfied simultaneously, but not so large so that (A3) is not satisfied. 
This can be achieved by adapting the dictionary to both dynamical systems
(1) and (2) \textit{simultaneously} and finding a common invariant subspace. One can perform this 
numerically via the recently developed dictionary learning based EDMD method\cite{li-2017},
and we demonstrate it through a numerical example in the next section. 
In essence, we rephrase the question into finding the best approximation of $h$ in
some adaptively chosen subspace $U$, which satisfies (A1),(A2), and (A4) \textit{approximately}. 
If $U$ is broad enough, this may well form a good approximation, at least on bounded domains. 
To make this statement rigorous, one will need to relax Theorem \ref{thm:edmd_m} into an ``approximate" version.
Lastly, we have assumed throughout that we have access to at least one ``matching data point'' satisfying some
non-degenerate conditions (A5). This stems from the fact that eigenfunctions (and projected eigenvectors)
are only unique up to a constant---which must be fixed in order to perform matching. We will return to this ``matching data point issue'' below.

\subsection{Numerical Example for EDMD-M}\label{sec:num example}
We now apply EDMD-M in a numerical example, to test and demonstrate that the method gracefully degrades when loosening the assumptions.
Consider the Van der Pol system with inverted stability, such that the limit cycle is repelling and the steady state is attracting:
\begin{eqnarray}
\label{eq:vdp1}\frac{d}{dt}x_1&=&-x_2,\\
\label{eq:vdp2}\frac{d}{dt}x_2&=&-\mu (1-x_1^2)x_2-x_1.
\end{eqnarray}
where we choose $\mu=1$.  We focus on a small domain around the steady state at $(0,0)$, away from the limit cycle. Then, a (possibly nonlinear) homeomorphism $h=(h_1,h_2)$ transforming states $(x_1(t),x_2(t))$ into
$(y_1(t),y_2(t))=(h_1(x_1(t)),h_2(x_2(t)))$ will yield trajectories of a topologically orbit equivalent system. For this example, we choose
\begin{eqnarray}
h_i(x_i) &=& \ln\left(a+\exp\left(b x_i\right)\right),\ a>0,b \in \mathbb{R}/\{0\},\\
h^{-1}_i(x_0) &=& \ln\left(\exp( x_i ) - a\right) /b.
\end{eqnarray}
We will again pretend to not know these transformations, and then use EDMD-M to (re)discover them. We only kept one matching data point to satisfy (A5).
Figure~\ref{fig:vdp transformed} shows an example trajectory at $a=1.2, b=-1.5$ of the original system defined through Eqs.~(\ref{eq:vdp1})--(\ref{eq:vdp2}) and its transformation through $h(x_1,x_2)=(h_1(x_1),h_2(x_2))$.
\begin{figure}[h]
\centering
\includegraphics[width=0.5\textwidth]{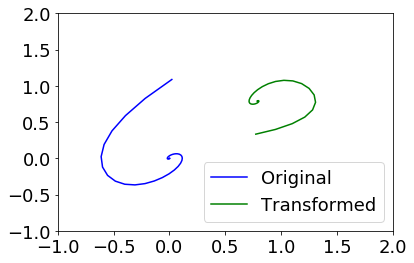}
\caption{\label{fig:vdp transformed}Original trajectory of the van der Pol system (blue) and the transformed trajectory (green), with parameters $a=1.2,b=-1.5$.}
\end{figure}
We now apply EDMD-M as outlined in the previous section to recover the transformation $h$ as an approximation
$\hat{h}$ from data of the van der Pol system and the system we want to match it to. 
The data is sampled independently, in an area around the steady state of the respective systems.
In the usual EDMD with dictionary learning \cite{li-2017}, we parameterize a dictionary set $\Psi$ by 
a multi-layer neural network, whose trainable weights are represented by $\theta$. That is,
$\Psi(\cdot;\theta)$ is a $\mathbb{R}^{N}$-valued function.
One can then proceed as in the usual EDMD algorithm and minimize the loss function
$J(K,\theta)=\sum_{x,x'} \| \Psi(x',\theta) - K \Psi(x,\theta) \|^2$ where
$x,x'$ are pairs of data where $x=x(t)$ and $x'=x(t+\Delta t)$ for a trajectory $x(t)$.
The main difference is that EDMD-M requires us to use the same dictionary $\Psi$ for both systems. 
Therefore, we define a combined loss function for both systems in the form
\begin{eqnarray}
J(K^{(1)},K^{(2)}&&,\theta)=\sum_{x',x,y',y} 
\|\Psi(x', \theta) - K^{(1)}\Psi(x, \theta)\|^2 +
\|\Psi(y', \theta) - K^{(2)}\Psi(y, \theta)\|^2,
\end{eqnarray}
where the sum is over all data points $x,x',y,y'\in\mathbb{R}^2$ of the form
$x'=x(t+\Delta t), x=x(t), y'=y(t+\Delta t), y=y(t)$ for some trajectories $x(t),y(t)$ for system (1)
and (2) respectively.
For the current application, we use a three-layer fully-connected neural network with tanh activation functions
to represent $\Psi$.
The learning phase of the network iterates between computing the matrices $K^{(1)},K^{(2)}$ by evaluating the current dictionary $\Psi$, and then updating the dictionary (the neural network) with fixed $K^{(1)},K^{(2)}$:
\begin{enumerate}
\item Compute $K^{(1)},K^{(2)}$ by assigning $K^{(i)}\leftarrow (G^{(i)}(\theta)+\lambda I)^{-1}A^{(i)}(\theta)$.
\item Update $\Psi(x;\theta)$ by changing $\theta\leftarrow \theta - \delta \nabla_\theta J(K^{(1)},K^{(2)},\theta)$.
\end{enumerate}
The matrices $K^{(1)},K^{(2)}$ obtained in step 1 will in general \textbf{not be similar}, i.e. will not have the same spectrum, even though they have the same singular values; we do not discuss this problem in detail, as it does not appear to matter for this example.
To direct the iterations to converge on matrices $K^{(1)},K^{(2)}$ which at least have the same singular values $\sigma_1,\dots,\sigma_n$, we introduce an intermediate optimization step, splitting step 1 into
\begin{enumerate}
\item[1. (a)] Compute $K^{(1)},K^{(2)}$ from $\theta$ as before;
\item[1. (b)] Use $K^{(1)}$, $K^{(2)}$ and a random, square matrix $P$ as initial conditions for
\begin{equation}
(K^{(1')},K^{(2')},P)=\min_{(A,B,P')} \lbrace\beta \|P' A- B P'\|^2+J(A,B,\theta)\rbrace;
\end{equation}
Here, $\beta$ is some regularization parameter, which we take to be 100. 
\item[1. (c)] Assign $K^{(1)}\leftarrow K^{(1')}$, $K^{(2)}\leftarrow K^{(2')}$.
\end{enumerate}
We solve the optimization problem in step 1.(b) with sequential least squares programming (SLSQP), with the constraint that $\Vert P\Vert_F^2\geq 1$ to remove the trivial, zero solution.
In fact, we should also require $P$ to be invertible in order to ensure similarity of $K^{(1')},K^{(2')}$.
However, empirically we found that the matrix $P$ resulting from the quadratic program above
is invertible and thus further regularization was not performed. 
 %
Figure~\ref{fig:vdp matching} shows a matched trajectory of the transformed system, computed through the approximated transformation $\hat{h}$ defined in theorem~\ref{thm:edmd-m}~(ii). We use a random point $(x_1,x_2)_0$ paired with $h((x_1,x_2)_0)$ as the required matching pair. 
Figure~\ref{fig:vdp matching relative errors} shows relative errors $\frac{|{h}_i(x(t))-\hat{h_i}(x(t))|}{|{h_i}(x(t))|}$  for both coordinates, averaged over 100 trajectories with random initial conditions. The relative error is about five percent for the initial conditions and then rapidly decays.
\begin{figure}[h!]
\centering
\includegraphics[width=0.9\textwidth]{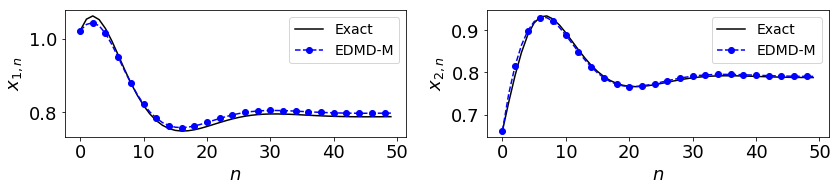}
\caption{\label{fig:vdp matching}An illustrative trajectory of the transformed system with $a=1.2,b=-1.5$, generated with the transformations $h$ and $\hat{h}$. The left plot shows the exact values from $h_1(x_1(t))$ overlayed with the values of $\hat{h}_1(x_1(t))$ from EDMD-M; the right plot shows $h_2(x_2(t))$ and $\hat{h}_2(x_2(t))$.}
\end{figure}
\begin{figure}[h!]
\centering
\includegraphics[width=0.45\textwidth]{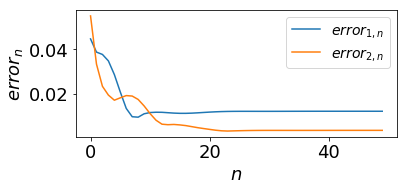}
\caption{\label{fig:vdp matching relative errors}Relative errors for both coordinates $x_1$ and $x_2$, averaged over 100 trajectories with randomly sampled initial conditions. The error is about five percent for the initial conditions and then rapidly decays.}
\end{figure}
\begin{figure}[h!]
\centering
\includegraphics[width=0.5\textwidth]{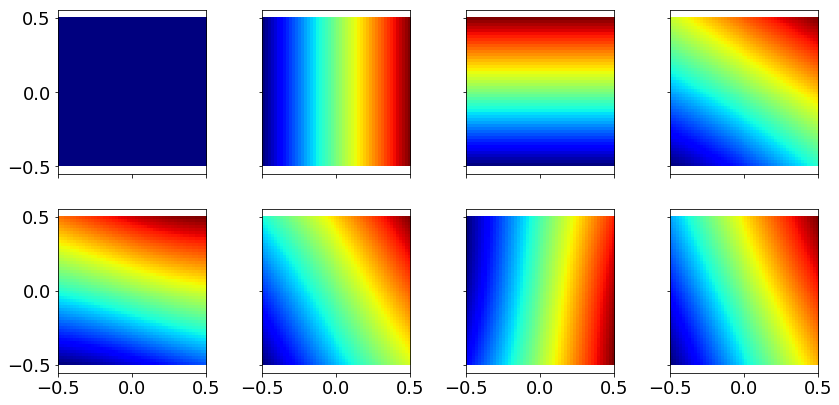}
\caption{\label{fig:vdp_matching_dictionary}The dictionary used in EDMD-M. The first three elements are the constant and two identity functions, the remaining five are obtained from the neural network.}
\end{figure}

\section{Conclusion}


Classifying dynamical systems up to an equivalence relationship is perhaps the foundational concept of the field.  In this paper we have demonstrated a way of inferring such transformations between dyamical systems, when they are so related.  We have described how the recently repopularized (but classical in egodic theory) spectral theory of Koopman operators offers a way to associate dynamical systems when an exact transformation between them exists.  There are several major obstacles to this program that we discussed in this preliminary work.  We have shown that the PDE associated with the infinitesimal generator of a Koopman (Composition) operator allows the explicit construction of these eigenfunctions in many cases.  Specifically, it is known that spectral matching is not sufficient for system matching in the general scenario: two systems can be related by a spectral isomorphism, but not have a spatial isomorphism\cite{redei2012history}. 
Even if there exists a spatial isomorphism, finding it through eigenfunctions is a challenge due to the geometric multiplicity of eigenvalues in higher dimensions. Despite this, we have shown that there are special cases where we can still use spectral matching to compute orbit equivalences through the matched eigenfunctions.  Beyond these special examples, we have then developed a computational optimization framework leading to what we call EDMD-M (Extended Dynamic Mode Decomposition Matching) that attains useful solutions.  Furthermore, this leads to a notion of approximate matching, that may be a useful alternative to the ``conjugacy defect'' arrived at via  fixed point iteration methods \cite{skufca2008concept, bollt2010comparing}.  

We already know that the flow box theorem guarantees that any vector field can be locally transformed to be the trivial flow $\dot{z}=(1,0\dots,0)^T$ away from singularities.
This implies that the transformation we try to construct through our methods exists locally, since matching each of our two candidate systems to the trivial flow also allows us to locally match them to each other. 
The major difficulty then lies in determining how far one can extend the domain and range of the resulting transformation, in other words, ``how big a patch can we match''. 
Clearly, this is limited by the same singularities that cause the global breakdown in the proof of  the flow box theorem.  
Promising ideas addressing the regularization of singularities in various branches of mathematics (from topological compactification~\cite{ozyesil-2016} to post-focusing PDE solutions~\cite{fibich-2011}) 
should provide both inspiration and possible road maps towards tackling these limitations.

\section{Acknowledgments}
We thank Igor Mezi\'{c} for his very helpful comments in the preparation of the document. The work of I.G.K. was partially supported by DARPA-MoDyL (HR0011-16-C-0116) and by the U.S. National Science Foundation (ECCS-1462241). I.G.K. and F.D. are grateful for the hospitality and support of the IAS-TUM. F.D. is also grateful for the support from the TopMath Graduate Center of TUM Graduate School at the Technical University of Munich, Germany, and from the TopMath Program at the Elite Network of Bavaria. E.M.B. thanks the Army Research Office (N68164-EG) and the Office of Naval Research (N00014-15-1-2093). Q.L. is grateful for the support of the Agency for Science, Technology and Research, Singapore.

\newpage
\appendix

\section{The PDE for Koopman eigenfunctions}\label{pdeforKoopmaneig}

The defining relation of a KEIGs pair, ${\cal K}_t[g](z)=e^{\lambda t} g(z)$, describes the rate of growth of an observation. Fig.~\ref{fig1} illustrates this behavior.  Inspection of the definition of evolution of $g$ along orbits suggests that we are demanding an equation for $g$ that satisfies,
\begin{equation}\label{n1}
\dot{g}=\lambda g.
\end{equation}
Recall that $g$ is a function, $g:M\rightarrow {\mathbb C}$. The notation $\dot{(\mbox{ })}=d(\mbox{ })/dt$ of the time derivative of $g$ is an abbreviation for the statement that the function values change over the trajectory $z(t)$ in the way given through Eq.~\ref{n1}. By the chain rule,
\begin{equation}\label{n2}
\dot{g}=d(g(z))/dt=\sum_i \frac{\partial g(z)}{\partial z_i} \dot{z}_i=\nabla g \cdot \dot{z}=\nabla g \cdot F(z).
\end{equation}
Hence, combining Eqs.~(\ref{n1}) and (\ref{n2}), we get (repeating Theorem \ref{thm2}):

\begin{thm}\label{thm2b}
 Given a domain ${\mathbb X}\subseteq M\subseteq{\mathbb R^d}$, $z\in{\mathbb X}$, and $\dot{z}=F(z)$ with $F:{\mathbb X}\rightarrow {\mathbb R}^d$, then the corresponding Koopman operator has eigenfunctions $g(z)$ that are solutions of the linear PDE,
\begin{equation}\label{qpdeb}
\nabla g \cdot F(z)=\lambda g(z).
\end{equation}
if $\mathbb{X}$ is compact and $g(z):{\mathbb X}\rightarrow {\mathbb C}$ is in $C^1({\mathbb X})$, or alternatively, if $g(z)$ is $C^2({\mathbb X})$. 
\end{thm} 
\begin{rem}
When the domain ${\mathbb X}$ is chosen too large, then solutions to (\ref{qpde}) may only exist in a weak sense,
\begin{equation}
-\int g\, \mathrm{div}(Fv) dx= \lambda \int g v dx
\end{equation} for test functions $v$ in an appropriately chosen function space. In this paper, we focus on strong solutions, restricting ${\mathbb X}$ where necessary.
\end{rem}

Equation~(\ref{qpde}),~(\ref{qpdeb}) follows more rigorously through a discussion of the infinitesimal generator of the Koopman operator.
This infinitesimal generator, which we denote $A_\mathcal{K}$, acts on observables $g$ such that
\begin{equation}\label{infeq}
A_\mathcal{K}g(x)=\mathop{\lim}_{t\rightarrow
0}\frac{g(S_t(x_0))-g(x_0)}{t}=\mathop{\lim}_{t\rightarrow
0}\frac{g(x(t))-g(x_0)}{t},
\end{equation}
which follows from the definition of the operator (Eq.~\ref{eq:koopman operator}).
If $g$ is continuously differentiable on a compact set $\mathbb{X}$, $g\in C^1({\mathbb X})$,  
we can apply the mean value theorem to obtain
\begin{eqnarray}
A_\mathcal{K}g(x)&=&\mathop{\sum}_{i=1}^{d}\frac{\partial g}{\partial
x_i}F_i(x)\\
\label{infkoopman}&=&\nabla g \cdot F(x).
\end{eqnarray}
Further details can be found in, \cite{lasota2013chaos, bollt2013applied}. Alternatively, for more general $\mathbb{X}$,  this also holds for $g\in C^2(\mathbb{X})$.
Now we complete the proof of Theorem \ref{thm2} (\ref{thm2b}), which is immediate.  Eq.~(\ref{n1}) together with Eq.~(\ref{infkoopman}) give the KEIGS linear PDE,
\begin{equation}\label{thm1b}
\nabla g \cdot F(x)=\lambda g(x).
\end{equation}
%
\medskip In Appendix~\ref{solns} we discuss several example solutions of this problem. Generally, it can be solved using the method of characteristics, when function values on an appropriately chosen initial curve are given.

%
%

\section{On Solutions of the Linear Koopman PDE}\label{solns}

The linear Koopman PDE, Eq.~(\ref{qpde}) of Theorem \ref{thm2}  may be solved by either the method of integrating factors in the case of one spatial domain (ODE), or by the method characteristics for higher dimensional domains (PDE), both of which we review here.  This will allow us to discuss the nature of the spectrum and eigenfunctions.  Most relevant to our discussion is the fact that the KEIGs may have nontrivial geometric multiplicity if $d>1$.

\subsection{One-Dimensional Spatial Domain, the ODE Case}

In the single-variable case,  for a given $F(x)$, the ODE from Theorem \ref{thm2} Eq.~(\ref{qpde}) becomes,
\begin{equation}
\frac{dg}{dx}F(x)=\lambda g, g(x_0)=g_0.
\end{equation}
 The standard method of multiplying factors is to define,
\begin{equation}
M(x)=e^{\int_{x_0}^x -\frac{\lambda}{F(s)}ds},
\end{equation}
from which a perfect derivative is developed,
\begin{equation}
\frac{d}{dx}(g(x)M(x))=0,
\end{equation}
and,
\begin{equation}
g(x)=e^{\int_{x_0}^x \frac{\lambda}{F(s)}ds}g_0,
\end{equation}
follows.  Notice that only the scalar constant $g_0$ differs, and we consider eigenfunctions different up to a scalar constant as an equivalence class. 

\subsection{The PDE Case, $d>1$}\label{pdecase}  

In the multivariate case, the PDE in Theorem \ref{thm2}, \ref{thm2b}, Eq.~(\ref{qpde}, \ref{qpdeb}),
has solutions that can be understood in terms of propogating initial data.  The problem may be solved by the method of characteristics, and it is straight forward (in our case) to derive that characteristic curves along which initial data propogates are the solutions of the underlying ODE.  Initial data may be defined on any co-dimension-one surface $\Sigma$ that is transverse to the flow.  Considering that on $\Sigma$, one may choose arbitrary $C^1(\Sigma)$ functions $g_0(z)$. For this reason, we see that there can be an uncountable number of solutions of the PDE, even for a single eigenvalue $\lambda$, and in a given domain.  This is the source of what we have called the nontrivial geometric multiplicity of Koopman eigenvalues in Sec.~\ref{defectsection}.  We will illustrate this phenomenon with the following two-dimensional examples.


\subsubsection{Two-Dimensional Spatial Domain, Example 1}

Consider the 2D nonlinear system
\begin{eqnarray*}
\dot{x} & = & y\qquad x(0)=x_{0}\\
\dot{y} & = & \frac{y^{2}}{x}\qquad y(0)=y_{0}
\end{eqnarray*}
for $(x,y)\in\mathbb{R}^{2}\backslash\left\{ 0\right\} $.
The solution is 
\begin{equation}
x(t)=\frac{x_{0}^{2}}{x_{0}-ty_{0}}\qquad y(t)=\frac{x_{0}^{2}y_{0}}{(x_{0}-ty_{0})^{2}}
\end{equation}
Now, consider the linear PDE defining the KEIGs
\begin{equation}
\nabla g(x,y)\cdot F(x,y)=\lambda g(x,y)
\end{equation}
where 
\begin{equation}
F(x,y)=(y,\frac{y^{2}}{x}).
\end{equation}
Let us set 
\begin{equation}
g(x,y)=e^{G(x,y)}.
\end{equation}
Then, we have
\begin{equation}
\nabla G(x,y)\cdot F(x,y)=\lambda.
\end{equation}
This is a linear PDE solvable by the method of characteristics. To obtain
unique solutions, we have to prescribe the function value of $G$
on some curve $\Sigma\subset\mathbb{R}^{2}$. Let us pick the curve $\Sigma:=\{(x,y): x=y\}$ (of
course, other curves can be chosen) and suppose $G(x,y)=G_{0}(x)$
on $\Sigma$ for some function $G_{0}:\mathbb{R}\rightarrow\mathbb{R}$.
Then, on the characteristic 
\begin{equation}
\frac{d}{ds}(x(s),y(s))=F(x(s),y(s))\label{eq:chars}
\end{equation}
we have 
\begin{equation}
\frac{d}{ds}G(s)\equiv\frac{d}{ds}G(x(s),y(s))=\lambda.\label{eq:evo_on_chars}
\end{equation}
We can solve (\ref{eq:chars}) and (\ref{eq:evo_on_chars}) to get
\begin{equation}
x(s)=\frac{x_{0}^{2}}{x_{0}-sy_{0}}\qquad y(s)=\frac{x_{0}^{2}y_{0}}{(x_{0}-sy_{0})^{2}}\qquad G(s)=\lambda s.
\end{equation}
It turns out in this case that the characteristics do not cross and
hence we can invert the expressions for $(x(s),y(s))$
to get 
\begin{equation}
x_{0}(x,y,s)=\frac{x^{2}}{x+sy}\qquad y_{0}(x,y,s)=\frac{x^{2}y}{(x+sy)^{2}}.
\end{equation}
Backtracking, we can find the time $s$ at which the characteristics
cross $\Sigma$. By setting $x_{0}=y_{0}$ to solve for $s=s(x,y)$,
we get
\begin{equation}
s(x,y)=\frac{y-x}{y}
\end{equation}
and so 
\begin{equation}
x_{0}(x,y,s(x,y))=y_{0}(x,y,s(x,y))=\frac{x^{2}}{y}.
\end{equation}
Therefore, we have the solution 
\begin{eqnarray*}
G(x,y) & = & G_{0}(x_{0}(x,y,s(x,y)))+\lambda s(x,y)\\
 & = & G_{0}(\frac{x^{2}}{y})+\lambda(\frac{y-x}{y}).
\end{eqnarray*}
and hence
\begin{equation}
g(x,y)=e^{G_{0}(\frac{x^{2}}{y})+\lambda(\frac{y-x}{y})}
\end{equation}
with $G_{0}:\mathbb{R}\rightarrow\mathbb{R}$ an arbitrary $C^{1}(\Sigma)$
function. In other words, we have a family of eigenfunctions indexed
by a function $G_{0}$, all with the same eigenvalue $\lambda$.

\subsubsection{Two-Dimensional Spatial Domain, Example 2}
\label{2dlinear}

As a second example, we take 
\begin{equation}
F(x,y)=(ax,by).
\end{equation}
As before, we can take $e^{G(x,y)}=g(x,y)$
and solve 
\begin{equation}
\nabla G\cdot F=\lambda.
\end{equation}
The characteristic equation solutions are 
\begin{equation}
x(s)=x_{0}e^{as}\qquad y(s)=y_{0}e^{bs}
\end{equation}
which inverts to 
\begin{equation}
x_{0}(x,y,s)=xe^{-as}\qquad y_{0}(x,y,s)=ye^{-bs}.
\end{equation}
It remains to define a curve on which we assign the function values.
We consider 
\begin{equation}
\Sigma=\left\{ x,y:x^{2}+\vert y\vert^{\frac{2a}{b}}=1\right\}.
\end{equation}
Then all characteristics emanating from $\Sigma$ (for $s\in \mathbb{R}$)  cover all
of $\mathbb{R}^{2}\backslash\left\{ 0\right\} $. In particular, we
find the crossing time $s(x,y)$ by solving $\left[x_{0}(x,y,s)\right]^{2}+\left|y_{0}(x,y,s)\right|^{\frac{2a}{b}}=1$,
which gives 
\begin{equation}
s(x,y)=\frac{1}{2a}\log(x^{2}+\vert y\vert^{\frac{2a}{b}})
\end{equation}
and 
\begin{equation}
x_{0}(x,y,s(x,y))=\frac{x}{\sqrt{\left|y\right|^{\frac{2a}{b}}+x^{2}}}\qquad y_{0}(x,y,s(x,y))=\frac{y}{(x^{2}+\left|y\right|^{\frac{2a}{b}})^{\frac{b}{2a}}}.
\end{equation}
Hence, the general solution is 
\begin{eqnarray*}
G(x,y) & = & G_{0}(x_{0}(x,y,s(x,y)),x_{y}(x,y,s(x,y)))+\lambda s(x,y)\\
 & = & \frac{\lambda}{2a}\log(x^{2}+\vert y\vert^{\frac{2a}{b}})+G_{0}(\frac{x}{\sqrt{x^{2}+\left|y\right|^{\frac{2a}{b}}}},\frac{y}{(x^{2}+\left|y\right|^{\frac{2a}{b}})^{\frac{b}{2a}}})
\end{eqnarray*}
which gives
\begin{equation}
g(x,y)=(x^{2}+\vert y\vert^{\frac{2a}{b}})^{\frac{\lambda}{2a}}g_{0}(\frac{x}{\sqrt{x^{2}+\left|y\right|^{\frac{2a}{b}}}},\frac{y}{(x^{2}+\left|y\right|^{\frac{2a}{b}})^{\frac{b}{2a}}})\label{eq:2D_2_soln}
\end{equation}
where $g_{0}:\Sigma\rightarrow\mathbb{R}$ is an arbitrary $C^{1}(\Sigma)$ function. 


\subsection{Numerical construction of the initial set}\label{sec:construction by numcont}

The solution to the PDE~\ref{pdecase} with the method of characteristics requires an appropriately chosen co-dimension-one initial set $\Sigma\subset\mathbb{X}$ in the state space, which is everywhere transverse to the flow.
The initial set can be a level set of an eigenfunction if the associated eigenvalue has a nonzero real part.  The requirement of starting with an eigenfunction to construct other eigenfunctions (you got to have money to make money!) through the PDE might seem restrictive, but Laplace Averages provide a means to find such a level set by numerical continuation, circumventing the effort to construct the initial eigenfunction over all of the state space.
Laplace Averages are eigenfunctions $f^*$ of the Koopman operator\cite{mauroy-2013}, and are defined when given an observation function $f:\mathbb{X}\to\mathbb{R}$, and the eigenvalue $\lambda$ associated with the eigenfunction. Then,
\begin{equation}
f^*(x):=\int_0^\infty (f\circ S_t)(x) \exp(-\lambda t) dt,
\end{equation}
if the integral exists.
To construct a level set of this eigenfunction, we proceed through numerical continuation, given a point $x_0$ in the state space $\mathbb{X}$.
First, we compute the absolute value of the Laplace average at the starting point, $|f^*(x_0)|=c$. This value fixes the level set $$L_c=\{x\in \mathbb{X}:|f^*(x)|=c\},$$ the connected part of which can now be constructed through numerical continuation.
For unstable fixed points, the Laplace average has to be computed backwards in time\cite{mauroy-2013}.
Fig.~\ref{fig:levelset_vdp} shows a level set, constructed through numerical continuation, close to the fixed point of the van der Pol system (see Eq.~\ref{eq:vdp1}) at $\mu=0.5$.
\begin{figure}[h!]
\centering
\includegraphics[width=0.6\textwidth]{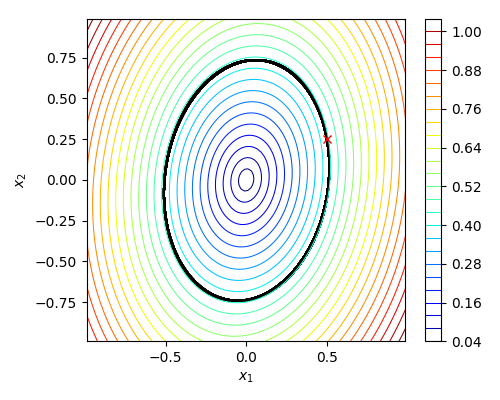}
\caption{\label{fig:levelset_vdp}Black line: Level set of an eigenfunction of the van der Pol system (Eq.~\ref{eq:vdp1}), computed through Laplace Averaging and numerical continuation, starting at the point $x_0$ (red cross). The colored contour lines show the level sets of the eigenfunction computed on the whole domain shown, to verify that the numerical continuation successfully found one of them.}
\end{figure}

\section{Example of Problems in Matching due to Nontrivial Geometric Multiplicity}
\label{defectexample}
The nontrivial geometric multiplicity of the Koopman eigenvalues poses a problem in the context of matching systems. Let us consider the example in Appendix \ref{2dlinear}. Suppose that we have two systems both with dynamics $F^{(1)}=F^{(2)}=[ax,by]^T$. We have, from before, the general solution of the eigenfunction
\begin{equation}
g(x,y)=(x^{2}+\vert y\vert^{\frac{2a}{b}})^{\frac{\lambda}{2a}}g_{0}(\frac{x}{\sqrt{x^{2}+\left|y\right|^{\frac{2a}{b}}}},\frac{y}{(x^{2}+\left|y\right|^{\frac{2a}{b}})^{\frac{b}{2a}}})\label{eq:2D_2_soln}
\end{equation}
where $g_{0}:\Sigma\rightarrow\mathbb{R}$ is an arbitrary $C^{1}(\Sigma)$ function. 
Now, notice that for any $g_{0}$, $g$ is an eigenfunction with eigenvalue
$\lambda$. Moreover, both dynamical systems $\left(1\right)$ and
$\left(2\right)$ have eigenfunctions of this form. Therefore, if
we ignore this nontrivial geometric multiplicity and pick, say, $g_{0}\left(x,y\right)=1$
for system $\left(1\right)$ and $g_{0}\left(x,y\right)=x$
for system $\left(2\right)$, we would have the stacked vector-valued
functions (c.f. Corr. 2)
\begin{eqnarray*}
	\mathcal{G}^{\left(1\right)}\left(x,y\right) & = & \left[\left(x^{2}+\vert y\vert^{\frac{2a}{b}}\right)^{\frac{\lambda_{1}}{2a}},\left(x^{2}+\vert y\vert^{\frac{2a}{b}}\right)^{\frac{\lambda_{2}}{2a}}\right]^{T},\\
	\mathcal{G}^{\left(2\right)}\left(x,y\right) & = & \left[\frac{x\left(x^{2}+\vert y\vert^{\frac{2a}{b}}\right)^{\frac{\lambda_{1}}{2a}}}{\sqrt{x^{2}+\left|y\right|^{\frac{2a}{b}}}},\frac{x\left(x^{2}+\vert y\vert^{\frac{2a}{b}}\right)^{\frac{\lambda_{2}}{2a}}}{\sqrt{x^{2}+\left|y\right|^{\frac{2a}{b}}}}\right]^{T}.
\end{eqnarray*}
Obviously, $\mathcal{G}^{\left(2\right),-1}\circ\mathcal{G}^{\left(1\right)}\left(x,y\right)\neq\left(x,y\right)$.

\bibliography{bibliography2}

\end{document}

%% file: f1a_vectorized.pdf_tex
\begingroup%
  \makeatletter%
  \providecommand\color[2][]{%
    \errmessage{(Inkscape) Color is used for the text in Inkscape, but the package 'color.sty' is not loaded}%
    \renewcommand\color[2][]{}%
  }%
  \providecommand\transparent[1]{%
    \errmessage{(Inkscape) Transparency is used (non-zero) for the text in Inkscape, but the package 'transparent.sty' is not loaded}%
    \renewcommand\transparent[1]{}%
  }%
  \providecommand\rotatebox[2]{#2}%
  \ifx\svgwidth\undefined%
    \setlength{\unitlength}{210.67099504bp}%
    \ifx\svgscale\undefined%
      \relax%
    \else%
      \setlength{\unitlength}{\unitlength * \real{\svgscale}}%
    \fi%
  \else%
    \setlength{\unitlength}{\svgwidth}%
  \fi%
  \global\let\svgwidth\undefined%
  \global\let\svgscale\undefined%
  \makeatother%
  \begin{picture}(1,0.93279803)%
    \put(0,0){\includegraphics[width=\unitlength,page=1]{f1a_vectorized.pdf}}%
    \put(0.85435308,0.35772549){\color[rgb]{0,0,0}\makebox(0,0)[lb]{\smash{$z_2$}}}%
    \put(0.36077929,0.89627041){\color[rgb]{0,0,0}\makebox(0,0)[lb]{\smash{$g$}}}%
    \put(0.00689212,0.00836708){\color[rgb]{0,0,0}\makebox(0,0)[lb]{\smash{$z_1$}}}%
    \put(0,0){\includegraphics[width=\unitlength,page=2]{f1a_vectorized.pdf}}%
    \put(0.58907595,0.36862549){\color[rgb]{0,0,0}\makebox(0,0)[lb]{\smash{$z(0)$}}}%
    \put(0.27914739,0.09017832){\color[rgb]{0,0,0}\makebox(0,0)[lb]{\smash{$z(t)$}}}%
    \put(0.72161094,0.86576144){\color[rgb]{0,0,0}\makebox(0,0)[lb]{\smash{$g(z(0))$}}}%
    \put(0.06554408,0.64325817){\color[rgb]{0,0,0}\makebox(0,0)[lb]{\smash{$g(z(t))$}}}%
    \put(0,0){\includegraphics[width=\unitlength,page=3]{f1a_vectorized.pdf}}%
  \end{picture}%
\endgroup%

%% file: f1b_vectorized.pdf_tex
\begingroup%
  \makeatletter%
  \providecommand\color[2][]{%
    \errmessage{(Inkscape) Color is used for the text in Inkscape, but the package 'color.sty' is not loaded}%
    \renewcommand\color[2][]{}%
  }%
  \providecommand\transparent[1]{%
    \errmessage{(Inkscape) Transparency is used (non-zero) for the text in Inkscape, but the package 'transparent.sty' is not loaded}%
    \renewcommand\transparent[1]{}%
  }%
  \providecommand\rotatebox[2]{#2}%
  \ifx\svgwidth\undefined%
    \setlength{\unitlength}{193.54421781bp}%
    \ifx\svgscale\undefined%
      \relax%
    \else%
      \setlength{\unitlength}{\unitlength * \real{\svgscale}}%
    \fi%
  \else%
    \setlength{\unitlength}{\svgwidth}%
  \fi%
  \global\let\svgwidth\undefined%
  \global\let\svgscale\undefined%
  \makeatother%
  \begin{picture}(1,1.01534157)%
    \put(0,0){\includegraphics[width=\unitlength,page=1]{f1b_vectorized.pdf}}%
    \put(0.84328113,0.19988742){\color[rgb]{0,0,0}\makebox(0,0)[lb]{\smash{$z(t)$}}}%
    \put(0.30603076,0.78608829){\color[rgb]{0,0,0}\makebox(0,0)[lb]{\smash{$g$}}}%
    \put(0.61437124,0.72796835){\color[rgb]{0,0,0}\makebox(0,0)[lb]{\smash{$g(z_0)$}}}%
    \put(0,0){\includegraphics[width=\unitlength,page=2]{f1b_vectorized.pdf}}%
    \put(-0.00391041,0.40525486){\color[rgb]{0,0,0}\makebox(0,0)[lb]{\smash{$g(z(t))$}}}%
    \put(0,0){\includegraphics[width=\unitlength,page=3]{f1b_vectorized.pdf}}%
  \end{picture}%
\endgroup%

%% file: main.bbl
\begin{thebibliography}{45}%
\makeatletter
\providecommand \@ifxundefined [1]{%
 \@ifx{#1\undefined}
}%
\providecommand \@ifnum [1]{%
 \ifnum #1\expandafter \@firstoftwo
 \else \expandafter \@secondoftwo
 \fi
}%
\providecommand \@ifx [1]{%
 \ifx #1\expandafter \@firstoftwo
 \else \expandafter \@secondoftwo
 \fi
}%
\providecommand \natexlab [1]{#1}%
\providecommand \enquote  [1]{``#1''}%
\providecommand \bibnamefont  [1]{#1}%
\providecommand \bibfnamefont [1]{#1}%
\providecommand \citenamefont [1]{#1}%
\providecommand \href@noop [0]{\@secondoftwo}%
\providecommand \href [0]{\begingroup \@sanitize@url \@href}%
\providecommand \@href[1]{\@@startlink{#1}\@@href}%
\providecommand \@@href[1]{\endgroup#1\@@endlink}%
\providecommand \@sanitize@url [0]{\catcode `\\12\catcode `\$12\catcode
  `\&12\catcode `\#12\catcode `\^12\catcode `\_12\catcode `\%12\relax}%
\providecommand \@@startlink[1]{}%
\providecommand \@@endlink[0]{}%
\providecommand \url  [0]{\begingroup\@sanitize@url \@url }%
\providecommand \@url [1]{\endgroup\@href {#1}{\urlprefix }}%
\providecommand \urlprefix  [0]{URL }%
\providecommand \Eprint [0]{\href }%
\providecommand \doibase [0]{http://dx.doi.org/}%
\providecommand \selectlanguage [0]{\@gobble}%
\providecommand \bibinfo  [0]{\@secondoftwo}%
\providecommand \bibfield  [0]{\@secondoftwo}%
\providecommand \translation [1]{[#1]}%
\providecommand \BibitemOpen [0]{}%
\providecommand \bibitemStop [0]{}%
\providecommand \bibitemNoStop [0]{.\EOS\space}%
\providecommand \EOS [0]{\spacefactor3000\relax}%
\providecommand \BibitemShut  [1]{\csname bibitem#1\endcsname}%
\let\auto@bib@innerbib\@empty
\bibitem [{\citenamefont {Mawhin}(1993)}]{mawhin1993centennial}%
  \BibitemOpen
  \bibfield  {author} {\bibinfo {author} {\bibfnamefont {J.}~\bibnamefont
  {Mawhin}},\ }\href@noop {} {\emph {\bibinfo {title} {The Centennial Legacy of
  Poincar{\'e} and Lyapunov in Ordinary Differential Equaitons}}}\ (\bibinfo
  {publisher} {Institut de math{\'e}matique pure et appliqu{\'e}e,
  Universit{\'e} catholique de Louvain},\ \bibinfo {year} {1993})\BibitemShut
  {NoStop}%
\bibitem [{\citenamefont {Skufca}\ and\ \citenamefont
  {Bollt}(2008)}]{skufca2008concept}%
  \BibitemOpen
  \bibfield  {author} {\bibinfo {author} {\bibfnamefont {J.~D.}\ \bibnamefont
  {Skufca}}\ and\ \bibinfo {author} {\bibfnamefont {E.~M.}\ \bibnamefont
  {Bollt}},\ }\bibfield  {title} {\enquote {\bibinfo {title} {A concept of
  homeomorphic defect for defining mostly conjugate dynamical systems},}\
  }\href@noop {} {\bibfield  {journal} {\bibinfo  {journal} {Chaos: An
  Interdisciplinary Journal of Nonlinear Science}\ }\textbf {\bibinfo {volume}
  {18}},\ \bibinfo {pages} {013118} (\bibinfo {year} {2008})}\BibitemShut
  {NoStop}%
\bibitem [{\citenamefont {Bollt}\ and\ \citenamefont
  {Skufca}(2010)}]{bollt2010comparing}%
  \BibitemOpen
  \bibfield  {author} {\bibinfo {author} {\bibfnamefont {E.~M.}\ \bibnamefont
  {Bollt}}\ and\ \bibinfo {author} {\bibfnamefont {J.~D.}\ \bibnamefont
  {Skufca}},\ }\bibfield  {title} {\enquote {\bibinfo {title} {On comparing
  dynamical systems by defective conjugacy: A symbolic dynamics interpretation
  of commuter functions},}\ }\href@noop {} {\bibfield  {journal} {\bibinfo
  {journal} {Physica D: Nonlinear Phenomena}\ }\textbf {\bibinfo {volume}
  {239}},\ \bibinfo {pages} {579--590} (\bibinfo {year} {2010})}\BibitemShut
  {NoStop}%
\bibitem [{\citenamefont {R{\'e}dei}\ and\ \citenamefont
  {Werndl}(2012)}]{redei2012history}%
  \BibitemOpen
  \bibfield  {author} {\bibinfo {author} {\bibfnamefont {M.}~\bibnamefont
  {R{\'e}dei}}\ and\ \bibinfo {author} {\bibfnamefont {C.}~\bibnamefont
  {Werndl}},\ }\bibfield  {title} {\enquote {\bibinfo {title} {{On the history
  of the isomorphism problem of dynamical systems with special regard to von
  Neumann's contribution}},}\ }\href@noop {} {\bibfield  {journal} {\bibinfo
  {journal} {Archive for history of exact sciences}\ }\textbf {\bibinfo
  {volume} {66}},\ \bibinfo {pages} {71--93} (\bibinfo {year}
  {2012})}\BibitemShut {NoStop}%
\bibitem [{\citenamefont {Neumann}(1932)}]{neumann1932operatorenmethode}%
  \BibitemOpen
  \bibfield  {author} {\bibinfo {author} {\bibfnamefont {J.~v.}\ \bibnamefont
  {Neumann}},\ }\bibfield  {title} {\enquote {\bibinfo {title} {{Zur
  Operatorenmethode in der klassischen Mechanik}},}\ }\href@noop {} {\bibfield
  {journal} {\bibinfo  {journal} {{Annals of Mathematics}}\ ,\ \bibinfo {pages}
  {587--642}} (\bibinfo {year} {1932})}\BibitemShut {NoStop}%
\bibitem [{\citenamefont {Halmos}\ and\ \citenamefont {von
  Neumann}(1942)}]{halmos1942operator}%
  \BibitemOpen
  \bibfield  {author} {\bibinfo {author} {\bibfnamefont {P.~R.}\ \bibnamefont
  {Halmos}}\ and\ \bibinfo {author} {\bibfnamefont {J.}~\bibnamefont {von
  Neumann}},\ }\bibfield  {title} {\enquote {\bibinfo {title} {Operator methods
  in classical mechanics, ii},}\ }\href@noop {} {\bibfield  {journal} {\bibinfo
   {journal} {Annals of Mathematics}\ ,\ \bibinfo {pages} {332--350}} (\bibinfo
  {year} {1942})}\BibitemShut {NoStop}%
\bibitem [{\citenamefont {Mezic}(2017)}]{mezic-2017}%
  \BibitemOpen
  \bibfield  {author} {\bibinfo {author} {\bibfnamefont {I.}~\bibnamefont
  {Mezic}},\ }\bibfield  {title} {\enquote {\bibinfo {title} {{Koopman Operator
  Spectrum and Data Analysis}},}\ }\href@noop {} {\bibfield  {journal}
  {\bibinfo  {journal} {arXiv}\ } (\bibinfo {year} {2017})},\ \Eprint
  {http://arxiv.org/abs/1702.07597v1} {1702.07597v1} \BibitemShut {NoStop}%
\bibitem [{\citenamefont {Lan}\ and\ \citenamefont
  {Mezi{\'c}}(2013)}]{lan2013linearization}%
  \BibitemOpen
  \bibfield  {author} {\bibinfo {author} {\bibfnamefont {Y.}~\bibnamefont
  {Lan}}\ and\ \bibinfo {author} {\bibfnamefont {I.}~\bibnamefont
  {Mezi{\'c}}},\ }\bibfield  {title} {\enquote {\bibinfo {title} {Linearization
  in the large of nonlinear systems and koopman operator spectrum},}\
  }\href@noop {} {\bibfield  {journal} {\bibinfo  {journal} {Physica D:
  Nonlinear Phenomena}\ }\textbf {\bibinfo {volume} {242}},\ \bibinfo {pages}
  {42--53} (\bibinfo {year} {2013})}\BibitemShut {NoStop}%
\bibitem [{\citenamefont {Mezi{\'c}}\ and\ \citenamefont
  {Banaszuk}(2004)}]{mezic2004comparison}%
  \BibitemOpen
  \bibfield  {author} {\bibinfo {author} {\bibfnamefont {I.}~\bibnamefont
  {Mezi{\'c}}}\ and\ \bibinfo {author} {\bibfnamefont {A.}~\bibnamefont
  {Banaszuk}},\ }\bibfield  {title} {\enquote {\bibinfo {title} {Comparison of
  systems with complex behavior},}\ }\href@noop {} {\bibfield  {journal}
  {\bibinfo  {journal} {Physica D: Nonlinear Phenomena}\ }\textbf {\bibinfo
  {volume} {197}},\ \bibinfo {pages} {101--133} (\bibinfo {year}
  {2004})}\BibitemShut {NoStop}%
\bibitem [{\citenamefont {Perko}(2013)}]{perko2013differential}%
  \BibitemOpen
  \bibfield  {author} {\bibinfo {author} {\bibfnamefont {L.}~\bibnamefont
  {Perko}},\ }\href@noop {} {\emph {\bibinfo {title} {Differential equations
  and dynamical systems}}},\ Vol.~\bibinfo {volume} {7}\ (\bibinfo  {publisher}
  {Springer Science \& Business Media},\ \bibinfo {year} {2013})\BibitemShut
  {NoStop}%
\bibitem [{\citenamefont {Mauroy}, \citenamefont {Mezi\'{c}},\ and\
  \citenamefont {Moehlis}(2013)}]{mauroy-2013}%
  \BibitemOpen
  \bibfield  {author} {\bibinfo {author} {\bibfnamefont {A.}~\bibnamefont
  {Mauroy}}, \bibinfo {author} {\bibfnamefont {I.}~\bibnamefont {Mezi\'{c}}}, \
  and\ \bibinfo {author} {\bibfnamefont {J.}~\bibnamefont {Moehlis}},\
  }\bibfield  {title} {\enquote {\bibinfo {title} {{Isostables, isochrons, and
  Koopman spectrum for the action--angle representation of stable fixed point
  dynamics}},}\ }\href {\doibase 10.1016/j.physd.2013.06.004} {\bibfield
  {journal} {\bibinfo  {journal} {Physica D: Nonlinear Phenomena}\ }\textbf
  {\bibinfo {volume} {261}},\ \bibinfo {pages} {19--30} (\bibinfo {year}
  {2013})}\BibitemShut {NoStop}%
\bibitem [{\citenamefont {Mohr}\ and\ \citenamefont
  {Mezi{\'c}}(2016)}]{mohr2016koopman}%
  \BibitemOpen
  \bibfield  {author} {\bibinfo {author} {\bibfnamefont {R.}~\bibnamefont
  {Mohr}}\ and\ \bibinfo {author} {\bibfnamefont {I.}~\bibnamefont
  {Mezi{\'c}}},\ }\bibfield  {title} {\enquote {\bibinfo {title} {Koopman
  principle eigenfunctions and linearization of diffeomorphisms},}\ }\href@noop
  {} {\bibfield  {journal} {\bibinfo  {journal} {arXiv preprint
  arXiv:1611.01209}\ } (\bibinfo {year} {2016})}\BibitemShut {NoStop}%
\bibitem [{\citenamefont {Mezi{\'{c}}}(2005)}]{mezic-2005}%
  \BibitemOpen
  \bibfield  {author} {\bibinfo {author} {\bibfnamefont {I.}~\bibnamefont
  {Mezi{\'{c}}}},\ }\bibfield  {title} {\enquote {\bibinfo {title} {{Spectral
  Properties of Dynamical Systems, Model Reduction and Decompositions}},}\
  }\href {\doibase 10.1007/s11071-005-2824-x} {\bibfield  {journal} {\bibinfo
  {journal} {Nonlinear Dynamics}\ }\textbf {\bibinfo {volume} {41}},\ \bibinfo
  {pages} {309--325} (\bibinfo {year} {2005})}\BibitemShut {NoStop}%
\bibitem [{\citenamefont {Budi{\v{s}}i{\'c}}, \citenamefont {Mohr},\ and\
  \citenamefont {Mezi{\'c}}(2012)}]{budivsic2012applied}%
  \BibitemOpen
  \bibfield  {author} {\bibinfo {author} {\bibfnamefont {M.}~\bibnamefont
  {Budi{\v{s}}i{\'c}}}, \bibinfo {author} {\bibfnamefont {R.}~\bibnamefont
  {Mohr}}, \ and\ \bibinfo {author} {\bibfnamefont {I.}~\bibnamefont
  {Mezi{\'c}}},\ }\bibfield  {title} {\enquote {\bibinfo {title} {{Applied
  Koopmanism}},}\ }\href@noop {} {\bibfield  {journal} {\bibinfo  {journal}
  {Chaos: An Interdisciplinary Journal of Nonlinear Science}\ }\textbf
  {\bibinfo {volume} {22}},\ \bibinfo {pages} {047510} (\bibinfo {year}
  {2012})}\BibitemShut {NoStop}%
\bibitem [{\citenamefont {Das}\ and\ \citenamefont
  {Giannakis}(2017)}]{das-2017}%
  \BibitemOpen
  \bibfield  {author} {\bibinfo {author} {\bibfnamefont {S.}~\bibnamefont
  {Das}}\ and\ \bibinfo {author} {\bibfnamefont {D.}~\bibnamefont
  {Giannakis}},\ }\bibfield  {title} {\enquote {\bibinfo {title}
  {{Delay-coordinate maps and the spectra of Koopman operators}},}\ }\href@noop
  {} {\bibfield  {journal} {\bibinfo  {journal} {arXiv}\ } (\bibinfo {year}
  {2017})},\ \Eprint {http://arxiv.org/abs/1706.08544v1} {1706.08544v1}
  \BibitemShut {NoStop}%
\bibitem [{\citenamefont {Giannakis}, \citenamefont {Slawinska},\ and\
  \citenamefont {Zhao}(2015)}]{giannakis-2015b}%
  \BibitemOpen
  \bibfield  {author} {\bibinfo {author} {\bibfnamefont {D.}~\bibnamefont
  {Giannakis}}, \bibinfo {author} {\bibfnamefont {J.}~\bibnamefont
  {Slawinska}}, \ and\ \bibinfo {author} {\bibfnamefont {Z.}~\bibnamefont
  {Zhao}},\ }\bibfield  {title} {\enquote {\bibinfo {title} {{Spatiotemporal
  feature extraction with data-driven Koopman operators}},}\ }\href@noop {}
  {\bibfield  {journal} {\bibinfo  {journal} {J. Mach. Learn. Res.
  Proceedings}\ ,\ \bibinfo {pages} {103--115}} (\bibinfo {year}
  {2015})}\BibitemShut {NoStop}%
\bibitem [{\citenamefont {Williams}\ \emph {et~al.}(2015)\citenamefont
  {Williams}, \citenamefont {Rowley}, \citenamefont {Mezi\'{c}},\ and\
  \citenamefont {Kevrekidis}}]{williams-2015b}%
  \BibitemOpen
  \bibfield  {author} {\bibinfo {author} {\bibfnamefont {M.~O.}\ \bibnamefont
  {Williams}}, \bibinfo {author} {\bibfnamefont {C.~W.}\ \bibnamefont
  {Rowley}}, \bibinfo {author} {\bibfnamefont {I.}~\bibnamefont {Mezi\'{c}}}, \
  and\ \bibinfo {author} {\bibfnamefont {I.~G.}\ \bibnamefont {Kevrekidis}},\
  }\bibfield  {title} {\enquote {\bibinfo {title} {{Data fusion via intrinsic
  dynamic variables: An application of data-driven Koopman spectral
  analysis}},}\ }\href {\doibase 10.1209/0295-5075/109/40007} {\bibfield
  {journal} {\bibinfo  {journal} {EPL (Europhysics Letters)}\ }\textbf
  {\bibinfo {volume} {109}} (\bibinfo {year} {2015}),\
  10.1209/0295-5075/109/40007}\BibitemShut {NoStop}%
\bibitem [{\citenamefont {Kemeth}\ \emph {et~al.}(2017)\citenamefont {Kemeth},
  \citenamefont {Haugland}, \citenamefont {Dietrich}, \citenamefont {Bertalan},
  \citenamefont {Li}, \citenamefont {Bollt}, \citenamefont {Talmon},
  \citenamefont {Krischer},\ and\ \citenamefont {Kevrekidis}}]{kemeth-2017}%
  \BibitemOpen
  \bibfield  {author} {\bibinfo {author} {\bibfnamefont {F.~P.}\ \bibnamefont
  {Kemeth}}, \bibinfo {author} {\bibfnamefont {S.~W.}\ \bibnamefont
  {Haugland}}, \bibinfo {author} {\bibfnamefont {F.}~\bibnamefont {Dietrich}},
  \bibinfo {author} {\bibfnamefont {T.}~\bibnamefont {Bertalan}}, \bibinfo
  {author} {\bibfnamefont {Q.}~\bibnamefont {Li}}, \bibinfo {author}
  {\bibfnamefont {E.~M.}\ \bibnamefont {Bollt}}, \bibinfo {author}
  {\bibfnamefont {R.}~\bibnamefont {Talmon}}, \bibinfo {author} {\bibfnamefont
  {K.}~\bibnamefont {Krischer}}, \ and\ \bibinfo {author} {\bibfnamefont
  {I.~G.}\ \bibnamefont {Kevrekidis}},\ }\bibfield  {title} {\enquote {\bibinfo
  {title} {An equal space for complex data with unknown internal order:
  Observability, gauge invariance and manifold learning},}\ }\href@noop {}
  {\bibfield  {journal} {\bibinfo  {journal} {arXiv}\ } (\bibinfo {year}
  {2017})},\ \Eprint {http://arxiv.org/abs/1708.05406v1} {1708.05406v1}
  \BibitemShut {NoStop}%
\bibitem [{\citenamefont {Williams}, \citenamefont {Kevrekidis},\ and\
  \citenamefont {Rowley}(2015)}]{williams-2015}%
  \BibitemOpen
  \bibfield  {author} {\bibinfo {author} {\bibfnamefont {M.~O.}\ \bibnamefont
  {Williams}}, \bibinfo {author} {\bibfnamefont {I.~G.}\ \bibnamefont
  {Kevrekidis}}, \ and\ \bibinfo {author} {\bibfnamefont {C.~W.}\ \bibnamefont
  {Rowley}},\ }\bibfield  {title} {\enquote {\bibinfo {title} {{A Data-Driven
  Approximation of the Koopman Operator: Extending Dynamic Mode
  Decomposition}},}\ }\href {\doibase 10.1007/s00332-015-9258-5} {\bibfield
  {journal} {\bibinfo  {journal} {J. Nonlinear Sci.}\ }\textbf {\bibinfo
  {volume} {25}},\ \bibinfo {pages} {1307--1346} (\bibinfo {year}
  {2015})}\BibitemShut {NoStop}%
\bibitem [{\citenamefont {Li}\ \emph {et~al.}(2017)\citenamefont {Li},
  \citenamefont {Dietrich}, \citenamefont {Bollt},\ and\ \citenamefont
  {Kevrekidis}}]{li-2017}%
  \BibitemOpen
  \bibfield  {author} {\bibinfo {author} {\bibfnamefont {Q.}~\bibnamefont
  {Li}}, \bibinfo {author} {\bibfnamefont {F.}~\bibnamefont {Dietrich}},
  \bibinfo {author} {\bibfnamefont {E.~M.}\ \bibnamefont {Bollt}}, \ and\
  \bibinfo {author} {\bibfnamefont {I.~G.}\ \bibnamefont {Kevrekidis}},\
  }\bibfield  {title} {\enquote {\bibinfo {title} {{Extended dynamic mode
  decomposition with dictionary learning: A data-driven adaptive spectral
  decomposition of the Koopman operator}},}\ }\href {\doibase
  10.1063/1.4993854} {\bibfield  {journal} {\bibinfo  {journal} {Chaos: An
  Interdisciplinary Journal of Nonlinear Science}\ }\textbf {\bibinfo {volume}
  {27}},\ \bibinfo {pages} {103111} (\bibinfo {year} {2017})}\BibitemShut
  {NoStop}%
\bibitem [{\citenamefont {Carmen}(2000)}]{carmen2000ordinary}%
  \BibitemOpen
  \bibfield  {author} {\bibinfo {author} {\bibfnamefont {C.}~\bibnamefont
  {Carmen}},\ }\bibfield  {title} {\enquote {\bibinfo {title} {Ordinary
  differential equations with applications},}\ }\href@noop {} {\bibfield
  {journal} {\bibinfo  {journal} {Texts in Applied Mathematics}\ }\textbf
  {\bibinfo {volume} {34}} (\bibinfo {year} {2000})}\BibitemShut {NoStop}%
\bibitem [{\citenamefont {Teschl}(2012)}]{teschl2012ordinary}%
  \BibitemOpen
  \bibfield  {author} {\bibinfo {author} {\bibfnamefont {G.}~\bibnamefont
  {Teschl}},\ }\href@noop {} {\emph {\bibinfo {title} {Ordinary differential
  equations and dynamical systems}}},\ Vol.\ \bibinfo {volume} {140}\ (\bibinfo
   {publisher} {American Mathematical Society Providence},\ \bibinfo {year}
  {2012})\BibitemShut {NoStop}%
\bibitem [{\citenamefont {Mezi{\'c}}(2013)}]{mezic2013analysis}%
  \BibitemOpen
  \bibfield  {author} {\bibinfo {author} {\bibfnamefont {I.}~\bibnamefont
  {Mezi{\'c}}},\ }\bibfield  {title} {\enquote {\bibinfo {title} {{Analysis of
  fluid flows via spectral properties of the Koopman operator}},}\ }\href@noop
  {} {\bibfield  {journal} {\bibinfo  {journal} {Annual Review of Fluid
  Mechanics}\ }\textbf {\bibinfo {volume} {45}},\ \bibinfo {pages} {357--378}
  (\bibinfo {year} {2013})}\BibitemShut {NoStop}%
\bibitem [{\citenamefont {Kutz}\ \emph {et~al.}(2016)\citenamefont {Kutz},
  \citenamefont {Brunton}, \citenamefont {Brunton},\ and\ \citenamefont
  {Proctor}}]{kutz2016dynamic}%
  \BibitemOpen
  \bibfield  {author} {\bibinfo {author} {\bibfnamefont {J.~N.}\ \bibnamefont
  {Kutz}}, \bibinfo {author} {\bibfnamefont {S.~L.}\ \bibnamefont {Brunton}},
  \bibinfo {author} {\bibfnamefont {B.~W.}\ \bibnamefont {Brunton}}, \ and\
  \bibinfo {author} {\bibfnamefont {J.~L.}\ \bibnamefont {Proctor}},\
  }\href@noop {} {\emph {\bibinfo {title} {Dynamic Mode Decomposition:
  Data-Driven Modeling of Complex Systems}}}\ (\bibinfo  {publisher} {SIAM},\
  \bibinfo {year} {2016})\BibitemShut {NoStop}%
\bibitem [{\citenamefont {Gaspard}\ \emph {et~al.}(1995)\citenamefont
  {Gaspard}, \citenamefont {Nicolis}, \citenamefont {Provata},\ and\
  \citenamefont {Tasaki}}]{gaspard1995spectral}%
  \BibitemOpen
  \bibfield  {author} {\bibinfo {author} {\bibfnamefont {P.}~\bibnamefont
  {Gaspard}}, \bibinfo {author} {\bibfnamefont {G.}~\bibnamefont {Nicolis}},
  \bibinfo {author} {\bibfnamefont {A.}~\bibnamefont {Provata}}, \ and\
  \bibinfo {author} {\bibfnamefont {S.}~\bibnamefont {Tasaki}},\ }\bibfield
  {title} {\enquote {\bibinfo {title} {Spectral signature of the pitchfork
  bifurcation: Liouville equation approach},}\ }\href@noop {} {\bibfield
  {journal} {\bibinfo  {journal} {Physical Review E}\ }\textbf {\bibinfo
  {volume} {51}},\ \bibinfo {pages} {74} (\bibinfo {year} {1995})}\BibitemShut
  {NoStop}%
\bibitem [{\citenamefont {Schmid}(2010)}]{schmid-2010}%
  \BibitemOpen
  \bibfield  {author} {\bibinfo {author} {\bibfnamefont {P.~J.}\ \bibnamefont
  {Schmid}},\ }\bibfield  {title} {\enquote {\bibinfo {title} {Dynamic mode
  decomposition of numerical and experimental data},}\ }\href {\doibase
  10.1017/s0022112010001217} {\bibfield  {journal} {\bibinfo  {journal}
  {Journal of Fluid Mechanics}\ }\textbf {\bibinfo {volume} {656}},\ \bibinfo
  {pages} {5--28} (\bibinfo {year} {2010})}\BibitemShut {NoStop}%
\bibitem [{\citenamefont {Williams}, \citenamefont {Rowley},\ and\
  \citenamefont {Kevrekidis}(2014)}]{williams-2014}%
  \BibitemOpen
  \bibfield  {author} {\bibinfo {author} {\bibfnamefont {M.~O.}\ \bibnamefont
  {Williams}}, \bibinfo {author} {\bibfnamefont {C.~W.}\ \bibnamefont
  {Rowley}}, \ and\ \bibinfo {author} {\bibfnamefont {I.~G.}\ \bibnamefont
  {Kevrekidis}},\ }\bibfield  {title} {\enquote {\bibinfo {title} {{A
  Kernel-Based Approach to Data-Driven Koopman Spectral Analysis}},}\
  }\href@noop {} {\bibfield  {journal} {\bibinfo  {journal} {arXiv}\ }\textbf
  {\bibinfo {volume} {v4}} (\bibinfo {year} {2014})}\BibitemShut {NoStop}%
\bibitem [{\citenamefont {Kaiser}, \citenamefont {Kutz},\ and\ \citenamefont
  {Brunton}(2017)}]{kaiser-2017}%
  \BibitemOpen
  \bibfield  {author} {\bibinfo {author} {\bibfnamefont {E.}~\bibnamefont
  {Kaiser}}, \bibinfo {author} {\bibfnamefont {J.~N.}\ \bibnamefont {Kutz}}, \
  and\ \bibinfo {author} {\bibfnamefont {S.~L.}\ \bibnamefont {Brunton}},\
  }\bibfield  {title} {\enquote {\bibinfo {title} {{Data-driven discovery of
  Koopman eigenfunctions for control}},}\ }\href@noop {} {\bibfield  {journal}
  {\bibinfo  {journal} {arXiv}\ } (\bibinfo {year} {2017})},\ \Eprint
  {http://arxiv.org/abs/1707.01146v1} {1707.01146v1} \BibitemShut {NoStop}%
\bibitem [{\citenamefont {Klus}\ \emph {et~al.}(2017)\citenamefont {Klus},
  \citenamefont {N\"{u}ske}, \citenamefont {Koltai}, \citenamefont {Wu},
  \citenamefont {Kevrekidis}, \citenamefont {Sch\"{u}tte},\ and\ \citenamefont
  {No\'{e}}}]{klus-2017}%
  \BibitemOpen
  \bibfield  {author} {\bibinfo {author} {\bibfnamefont {S.}~\bibnamefont
  {Klus}}, \bibinfo {author} {\bibfnamefont {F.}~\bibnamefont {N\"{u}ske}},
  \bibinfo {author} {\bibfnamefont {P.}~\bibnamefont {Koltai}}, \bibinfo
  {author} {\bibfnamefont {H.}~\bibnamefont {Wu}}, \bibinfo {author}
  {\bibfnamefont {I.}~\bibnamefont {Kevrekidis}}, \bibinfo {author}
  {\bibfnamefont {C.}~\bibnamefont {Sch\"{u}tte}}, \ and\ \bibinfo {author}
  {\bibfnamefont {F.}~\bibnamefont {No\'{e}}},\ }\bibfield  {title} {\enquote
  {\bibinfo {title} {Data-driven model reduction and transfer operator
  approximation},}\ }\href@noop {} {\bibfield  {journal} {\bibinfo  {journal}
  {arXiv}\ } (\bibinfo {year} {2017})},\ \Eprint
  {http://arxiv.org/abs/1703.10112v1} {1703.10112v1} \BibitemShut {NoStop}%
\bibitem [{\citenamefont {Golub}\ and\ \citenamefont
  {Van~Loan}(2012)}]{golub2012matrix}%
  \BibitemOpen
  \bibfield  {author} {\bibinfo {author} {\bibfnamefont {G.~H.}\ \bibnamefont
  {Golub}}\ and\ \bibinfo {author} {\bibfnamefont {C.~F.}\ \bibnamefont
  {Van~Loan}},\ }\href@noop {} {\emph {\bibinfo {title} {Matrix
  computations}}},\ Vol.~\bibinfo {volume} {3}\ (\bibinfo  {publisher} {JHU
  Press},\ \bibinfo {year} {2012})\BibitemShut {NoStop}%
\bibitem [{\citenamefont {Kelley}(1967)}]{kelley1967stable}%
  \BibitemOpen
  \bibfield  {author} {\bibinfo {author} {\bibfnamefont {A.}~\bibnamefont
  {Kelley}},\ }\bibfield  {title} {\enquote {\bibinfo {title} {The stable,
  center-stable, center, center-unstable, unstable manifolds},}\ }\href@noop {}
  {\bibfield  {journal} {\bibinfo  {journal} {Journal of Differential
  Equations}\ }\textbf {\bibinfo {volume} {3}},\ \bibinfo {pages} {546--570}
  (\bibinfo {year} {1967})}\BibitemShut {NoStop}%
\bibitem [{\citenamefont {Chicone}(2006)}]{marsden2006texts}%
  \BibitemOpen
  \bibfield  {author} {\bibinfo {author} {\bibfnamefont {C.}~\bibnamefont
  {Chicone}},\ }\href {\doibase 10.1007/b97645} {\emph {\bibinfo {title}
  {Ordinary Differential Equations with Applications}}}\ (\bibinfo  {publisher}
  {Springer},\ \bibinfo {year} {2006})\BibitemShut {NoStop}%
\bibitem [{\citenamefont {Tsiligiannis}\ and\ \citenamefont
  {Lyberatos}(1987)}]{tsiligiannis-1987}%
  \BibitemOpen
  \bibfield  {author} {\bibinfo {author} {\bibfnamefont {C.}~\bibnamefont
  {Tsiligiannis}}\ and\ \bibinfo {author} {\bibfnamefont {G.}~\bibnamefont
  {Lyberatos}},\ }\bibfield  {title} {\enquote {\bibinfo {title} {Steady state
  bifurcations and exact multiplicity conditions via {Carleman}
  linearization},}\ }\href {\doibase 10.1016/0022-247X(87)90082-5} {\bibfield
  {journal} {\bibinfo  {journal} {Journal of Mathematical Analysis and
  Applications}\ }\textbf {\bibinfo {volume} {126}},\ \bibinfo {pages}
  {143--160} (\bibinfo {year} {1987})}\BibitemShut {NoStop}%
\bibitem [{\citenamefont {Carleman}(1932)}]{carleman-1932}%
  \BibitemOpen
  \bibfield  {author} {\bibinfo {author} {\bibfnamefont {T.}~\bibnamefont
  {Carleman}},\ }\bibfield  {title} {\enquote {\bibinfo {title} {Application de
  la th{\'{e}}orie des {\'{e}}quations int{\'{e}}grales lin{\'{e}}aires aux
  syst{\`{e}}mes d'{\'{e}}quations diff{\'{e}}rentielles non
  lin{\'{e}}aires},}\ }\href {\doibase 10.1007/BF02546499} {\bibfield
  {journal} {\bibinfo  {journal} {Acta Mathematica}\ }\textbf {\bibinfo
  {volume} {59}},\ \bibinfo {pages} {63--87} (\bibinfo {year}
  {1932})}\BibitemShut {NoStop}%
\bibitem [{\citenamefont {Banks}(1992)}]{banks-1992}%
  \BibitemOpen
  \bibfield  {author} {\bibinfo {author} {\bibfnamefont {S.~P.}\ \bibnamefont
  {Banks}},\ }\bibfield  {title} {\enquote {\bibinfo {title}
  {{Infinite-dimensional Carleman linearization, the Lie series and optimal
  control of non-linear partial differential equations}},}\ }\href {\doibase
  10.1080/00207729208949241} {\bibfield  {journal} {\bibinfo  {journal}
  {International Journal of Systems Science}\ }\textbf {\bibinfo {volume}
  {23}},\ \bibinfo {pages} {663--675} (\bibinfo {year} {1992})}\BibitemShut
  {NoStop}%
\bibitem [{\citenamefont {Kowalski}\ and\ \citenamefont
  {Steeb}(1991)}]{kowalski-1991}%
  \BibitemOpen
  \bibfield  {author} {\bibinfo {author} {\bibfnamefont {K.}~\bibnamefont
  {Kowalski}}\ and\ \bibinfo {author} {\bibfnamefont {W.-H.}\ \bibnamefont
  {Steeb}},\ }\href
  {http://www.ebook.de/de/product/4355844/nonlinear_dynamical_systems_and_carleman.html}
  {\emph {\bibinfo {title} {{Nonlinear Dynamical Systems and Carleman}}}}\
  (\bibinfo  {publisher} {World Scientific},\ \bibinfo {year}
  {1991})\BibitemShut {NoStop}%
\bibitem [{\citenamefont {Brunton}, \citenamefont {Proctor},\ and\
  \citenamefont {Kutz}(2016)}]{brunton2016discovering}%
  \BibitemOpen
  \bibfield  {author} {\bibinfo {author} {\bibfnamefont {S.~L.}\ \bibnamefont
  {Brunton}}, \bibinfo {author} {\bibfnamefont {J.~L.}\ \bibnamefont
  {Proctor}}, \ and\ \bibinfo {author} {\bibfnamefont {J.~N.}\ \bibnamefont
  {Kutz}},\ }\bibfield  {title} {\enquote {\bibinfo {title} {Discovering
  governing equations from data by sparse identification of nonlinear dynamical
  systems},}\ }\href@noop {} {\bibfield  {journal} {\bibinfo  {journal}
  {Proceedings of the National Academy of Sciences}\ }\textbf {\bibinfo
  {volume} {113}},\ \bibinfo {pages} {3932--3937} (\bibinfo {year}
  {2016})}\BibitemShut {NoStop}%
\bibitem [{\citenamefont {Napoletani}\ and\ \citenamefont
  {Sauer}(2008)}]{napoletani2008reconstructing}%
  \BibitemOpen
  \bibfield  {author} {\bibinfo {author} {\bibfnamefont {D.}~\bibnamefont
  {Napoletani}}\ and\ \bibinfo {author} {\bibfnamefont {T.~D.}\ \bibnamefont
  {Sauer}},\ }\bibfield  {title} {\enquote {\bibinfo {title} {Reconstructing
  the topology of sparsely connected dynamical networks},}\ }\href@noop {}
  {\bibfield  {journal} {\bibinfo  {journal} {Physical Review E}\ }\textbf
  {\bibinfo {volume} {77}},\ \bibinfo {pages} {026103} (\bibinfo {year}
  {2008})}\BibitemShut {NoStop}%
\bibitem [{\citenamefont {Wang}, \citenamefont {Lai},\ and\ \citenamefont
  {Grebogi}(2016)}]{wang2016data}%
  \BibitemOpen
  \bibfield  {author} {\bibinfo {author} {\bibfnamefont {W.-X.}\ \bibnamefont
  {Wang}}, \bibinfo {author} {\bibfnamefont {Y.-C.}\ \bibnamefont {Lai}}, \
  and\ \bibinfo {author} {\bibfnamefont {C.}~\bibnamefont {Grebogi}},\
  }\bibfield  {title} {\enquote {\bibinfo {title} {Data based identification
  and prediction of nonlinear and complex dynamical systems},}\ }\href@noop {}
  {\bibfield  {journal} {\bibinfo  {journal} {Physics Reports}\ }\textbf
  {\bibinfo {volume} {644}},\ \bibinfo {pages} {1--76} (\bibinfo {year}
  {2016})}\BibitemShut {NoStop}%
\bibitem [{\citenamefont {Yao}\ and\ \citenamefont
  {Bollt}(2007)}]{yao2007modeling}%
  \BibitemOpen
  \bibfield  {author} {\bibinfo {author} {\bibfnamefont {C.}~\bibnamefont
  {Yao}}\ and\ \bibinfo {author} {\bibfnamefont {E.~M.}\ \bibnamefont
  {Bollt}},\ }\bibfield  {title} {\enquote {\bibinfo {title} {Modeling and
  nonlinear parameter estimation with kronecker product representation for
  coupled oscillators and spatiotemporal systems},}\ }\href@noop {} {\bibfield
  {journal} {\bibinfo  {journal} {Physica D: Nonlinear Phenomena}\ }\textbf
  {\bibinfo {volume} {227}},\ \bibinfo {pages} {78--99} (\bibinfo {year}
  {2007})}\BibitemShut {NoStop}%
\bibitem [{\citenamefont {Skufca}\ and\ \citenamefont
  {Bollt}(2007)}]{skufca2007relaxing}%
  \BibitemOpen
  \bibfield  {author} {\bibinfo {author} {\bibfnamefont {J.~D.}\ \bibnamefont
  {Skufca}}\ and\ \bibinfo {author} {\bibfnamefont {E.~M.}\ \bibnamefont
  {Bollt}},\ }\bibfield  {title} {\enquote {\bibinfo {title} {Relaxing
  conjugacy to fit modeling in dynamical systems},}\ }\href@noop {} {\bibfield
  {journal} {\bibinfo  {journal} {Physical Review E}\ }\textbf {\bibinfo
  {volume} {76}},\ \bibinfo {pages} {026220} (\bibinfo {year}
  {2007})}\BibitemShut {NoStop}%
\bibitem [{\citenamefont {Ozyesil}, \citenamefont {Sharon},\ and\ \citenamefont
  {Singer}(2016)}]{ozyesil-2016}%
  \BibitemOpen
  \bibfield  {author} {\bibinfo {author} {\bibfnamefont {O.}~\bibnamefont
  {Ozyesil}}, \bibinfo {author} {\bibfnamefont {N.}~\bibnamefont {Sharon}}, \
  and\ \bibinfo {author} {\bibfnamefont {A.}~\bibnamefont {Singer}},\
  }\bibfield  {title} {\enquote {\bibinfo {title} {{Synchronization over Cartan
  motion groups via contraction}},}\ }\href@noop {} {\bibfield  {journal}
  {\bibinfo  {journal} {arXiv}\ } (\bibinfo {year} {2016})},\ \Eprint
  {http://arxiv.org/abs/1612.00059v3} {1612.00059v3} \BibitemShut {NoStop}%
\bibitem [{\citenamefont {Fibich}\ and\ \citenamefont
  {Klein}(2011)}]{fibich-2011}%
  \BibitemOpen
  \bibfield  {author} {\bibinfo {author} {\bibfnamefont {G.}~\bibnamefont
  {Fibich}}\ and\ \bibinfo {author} {\bibfnamefont {M.}~\bibnamefont {Klein}},\
  }\bibfield  {title} {\enquote {\bibinfo {title} {{Continuations of the
  nonlinear Schr\"{o}dinger equation beyond the singularity}},}\ }\href
  {\doibase 10.1088/0951-7715/24/7/006} {\bibfield  {journal} {\bibinfo
  {journal} {Nonlinearity}\ }\textbf {\bibinfo {volume} {24}},\ \bibinfo
  {pages} {2003--2045} (\bibinfo {year} {2011})}\BibitemShut {NoStop}%
\bibitem [{\citenamefont {Lasota}\ and\ \citenamefont
  {Mackey}(2013)}]{lasota2013chaos}%
  \BibitemOpen
  \bibfield  {author} {\bibinfo {author} {\bibfnamefont {A.}~\bibnamefont
  {Lasota}}\ and\ \bibinfo {author} {\bibfnamefont {M.~C.}\ \bibnamefont
  {Mackey}},\ }\href@noop {} {\emph {\bibinfo {title} {Chaos, fractals, and
  noise: stochastic aspects of dynamics}}},\ Vol.~\bibinfo {volume} {97}\
  (\bibinfo  {publisher} {Springer Science \& Business Media},\ \bibinfo {year}
  {2013})\BibitemShut {NoStop}%
\bibitem [{\citenamefont {Bollt}\ and\ \citenamefont
  {Santitissadeekorn}(2013)}]{bollt2013applied}%
  \BibitemOpen
  \bibfield  {author} {\bibinfo {author} {\bibfnamefont {E.~M.}\ \bibnamefont
  {Bollt}}\ and\ \bibinfo {author} {\bibfnamefont {N.}~\bibnamefont
  {Santitissadeekorn}},\ }\href@noop {} {\emph {\bibinfo {title} {Applied and
  Computational Measurable Dynamics}}}\ (\bibinfo  {publisher} {SIAM},\
  \bibinfo {year} {2013})\BibitemShut {NoStop}%
\end{thebibliography}%
